\documentclass[reqno,oneside,11pt]{amsart}
\usepackage[T1]{fontenc}
\usepackage[utf8]{inputenc}
\usepackage{lmodern}
\usepackage[french,english]{babel}
\usepackage{geometry}

\usepackage{amsmath,amssymb,amsfonts,amsthm,mathtools}
\usepackage{mathrsfs}
\usepackage{graphicx}
\usepackage{caption}
\usepackage{paralist}
\usepackage{fancyhdr}
\usepackage{emptypage}
\fancypagestyle{plain}{
	\fancyhf{}
	\fancyfoot[RO,RE]{\thepage}
	
	}

\usepackage[noadjust]{cite}
\usepackage[
pagebackref=true,
colorlinks=true,
urlcolor=purple,
linkcolor=purple!87!black,
citecolor=green!60!black,
pdfborder={0 0 0}
]{hyperref}
\renewcommand*{\backref}[1]{}
\renewcommand*{\backrefalt}[4]{{\tiny(%
		\ifcase #1 Not cited.%
		\or Cited on page~#2.%
		\else Cited on pages #2.%
		\fi%
		)}}
\usepackage{cleveref}

\usepackage{bookmark}
\usepackage{dsfont}

\newcommand{\N}{\mathbb{N}}
\newcommand{\Z}{\mathbb{Z}}
\newcommand{\R}{\mathbb{R}}
\renewcommand{\P}{\mathbb{P}}
\newcommand{\E}{\mathbb{E}}
\newcommand{\supp}[1]{\mathrm{supp}(#1)}

\usepackage{tikz}
\usetikzlibrary{patterns,positioning,arrows,decorations.markings,calc,decorations.pathmorphing,decorations.pathreplacing}
\usetikzlibrary{fadings}
\usepackage{forest}
\usepackage{tikz-qtree}
\RequirePackage{pgffor} 
\makeatletter
\newcommand{\neutralize}[1]{\expandafter\let\csname c@#1\endcsname\count@}
\makeatother

\newcommand{\thistheoremname}{}

\newtheorem*{genericthm*}{\thistheoremname}
\newenvironment{namedthm*}[1]
{\renewcommand{\thistheoremname}{#1}%
	\begin{genericthm*}}
	{\end{genericthm*}}
\theoremstyle{plain}
\newtheorem{thm}{Theorem}[section] 
\newtheorem{proposition}[thm]{Proposition}
\newtheorem{lemma}[thm]{Lemma}

\newtheorem*{claim*}{Claim}
\newtheorem{corollary}[thm]{Corollary}
\theoremstyle{definition}
\newtheorem{rem}[thm]{Remark}
\newtheorem*{rem*}{Remark}
\newtheorem{definition}[thm]{Definition} 
\newtheorem{question}[thm]{Question}

\def\N{{\mathbb N}}

\def\P{{\mathbb P}}

\def\R{{\mathbb R}}
\def\Z{{\mathbb Z}}

\def\cC{{\mathcal C}}

\let\epsilon\varepsilon
\newcommand{\dd}{\mathop{}\!\mathrm{d}}
\DeclarePairedDelimiter\norm{\lVert}{\rVert}
\DeclarePairedDelimiter\abs{\lvert}{\rvert}
\makeatletter
\let\oldabs\abs
\def\abs{\@ifstar{\oldabs}{\oldabs*}}
\let\oldnorm\norm
\def\norm{\@ifstar{\oldnorm}{\oldnorm*}}
\makeatother

\newcommand\restr[2]{{
		\left.\kern-\nulldelimiterspace 
		#1 
		\vphantom{\big|} 
		\right|_{#2} 
}}

\newtheorem{thmA}{Theorem}

\allowdisplaybreaks

\DeclareRobustCommand\lind{l}
\allowdisplaybreaks
\usepackage{url}

\usepackage{enumitem}
\setlist[enumerate, 1]{label = (\roman*), ref = \roman*}
\setlist[enumerate, 2]{label = \theenumi.\alph*}

\author{Martín Gilabert Vio} 
\address{Institut Camille Jordan, Université Claude Bernard Lyon 1, 43 boulevard du 11 novembre
	1918, 69622 Villeurbanne, France}
\email[Martín Gilabert Vio]{gilabert@math.univ-lyon1.fr \\ martingilabertvio@gmail.com}
\author{Cosmas Kravaris}
\address{Department of Mathematics, Princeton University, Princeton, NJ, USA}
\email[Cosmas Kravaris]{ck6221@princeton.edu}
\author{Eduardo Silva} 
\address{University of Münster, Einsteinstrasse 62, Münster 48149, Germany}
\email[Eduardo Silva]{eduardo.silva@uni-muenster.de, edosilvamuller@gmail.com}
\urladdr{\url{https://edoasd.github.io/eduardo_silva_math/}}
\title[The Poisson boundary of Thompson's group $T$ is not the circle]{The Poisson boundary of Thompson's group $T$ is not the circle}
\begin{document}

	\maketitle
	\begin{abstract} 
	Let $\mu$ be a nondegenerate probability measure with finite entropy on a countable group $G \leq \mathrm{Homeo}_+(S^1)$ of orientation-preserving homeomorphisms of the circle acting proximally, minimally and topologically nonfreely on $S^1$. We prove that the circle $S^1$ endowed with its unique $\mu$-stationary probability measure is not the Poisson boundary of $(G,\mu)$. When $G$ is Thompson's group $T$ and $\mu$ is finitely supported, this answers a question posed by B. Deroin \cite{Deroin2013} and A. Navas \cite{Navas2018}.
\end{abstract}
\section{Introduction}
Countable subgroups of the group $\mathrm{Homeo}_+(S^1)$ of orientation-preserving homeomorphisms of the circle $S^1$ are a well-studied family of groups that are connected with the theory of circularly-ordered groups \cite{BaikSamperton2018}, left-ordered groups \cite{DeroinNavasRivas2016}, foliations of $3$-manifolds \cite{Calegari2007}, and bounded cohomology \cite{Ghys1987}. We refer to \cite{Ghys2001, Navas2011} for many examples of groups acting on the circle and an overview of the fundamentals of the subject.

Thompson's group $T$ is the group of dyadic piecewise affine homeomorphisms of the circle $S^1\cong \R/\mathbb{Z}$. That is, an orientation-preserving homeomorphism $g \colon S^1\to S^1$ belongs to $T$ if there exists a finite subset of dyadic rationals $D \subset \Z[1/2]/\Z$ such that $g$ restricted to every connected component $C$ of $S^1 \setminus D$ is of the form $g(x) = 2^k x + b$, $x\in S^1$, for some $k\in \Z$ and $b\in \Z[1/2]/\Z$. This group was introduced by R. Thompson in unpublished notes \cite{Thompson} as the first example of a finitely presented infinite simple group, and has been extensively studied from algebraic, dynamical, and cohomological viewpoints, see, e.g.,  \cite{Brin1996,  GhysSergiescu1987, BrownGeoghegan1984}.

In this paper we focus on random walks on subgroups of $\mathrm{Homeo}_+(S^1)$, and our results below apply in particular to Thompson's group $T$. Given a probability measure $\mu$ on a countable group $G$, the \emph{(right) $\mu$-random walk on $G$} is the Markov chain $(w_n)_{n\ge 0}$ with state space $G$, transition probabilities $\mathbb{P}(w_{n+1}=g\mid w_n=h)=\mu(g^{-1}h)$, for each $g,h\in G$ and $n\ge 0$, and initial state $w_0=e_G$.
In what follows we work under the assumption that $\mu$ is \emph{nondegenerate}, meaning that the semigroup generated by its support coincides with the group $G$.

The \emph{Poisson boundary} of the pair $(G,\mu)$ is a probability space $(\partial_{\mu}G,\nu)$ that captures all the stochastically significant asymptotic behavior of sample paths of the $\mu$-random walk on $G$, and is defined as follows.  A \emph{$\mu$-boundary} of $G$ is a probability space $(X, \nu)$ equipped with a measurable $G$-action and a shift-invariant \emph{boundary map} $\xi \colon G^\N \to X$ on the space of trajectories $(G^{\N},\mathbb{P})$ of the $\mu$-random walk, such that $\xi_\ast \P = \nu$. The Poisson boundary $(\partial_{\mu}G,\nu)$ is the \emph{maximal} $\mu$-boundary, in the sense that any $\mu$-boundary is a $G$-equivariant measurable quotient of it, and is unique up to $G$-equivariant isomorphisms. See also Definition \ref{defn: Poisson boundary original def} for an equivalent definition of the Poisson boundary. The pair $(\partial_\mu G, \nu)$ is a probability space equipped with a measure-class preserving action of $G$ such that $\nu$ is $\mu$-stationary, that is, the equation $\nu=\mu*\nu=\sum_{g\in G}\mu(g)g_{*}\nu$ holds. Recall that a function $f \colon G \to \R$ is \emph{$\mu$-harmonic} if 
\(
f(g) = \sum_{h \in G} f(gh) \mu(h)
\) for all $g \in G$. The Poisson boundary encodes the space $H^{\infty}(G,\mu)$ of bounded $\mu$-harmonic functions on $G$ via the so-called \emph{Poisson transform} $L^{\infty}(\partial_{\mu}G,\nu)\to H^{\infty}(G,\mu)$ given by
\begin{equation}\label{eq: Poisson transform}
	\begin{aligned}
		F&\mapsto \left( f(g)\coloneqq \int_{\partial_{\mu} G}F(gx)\ d\nu(x), \text{ for } g\in G\right).
	\end{aligned}
\end{equation}

An action of a countable group $G$ on $S^1$ by homeomorphisms is called
\begin{itemize}
	\item \emph{minimal} if all the $G$-orbits are dense in $S^1$,
	\item \emph{proximal} if for every proper closed interval $I \subsetneqq S^1$ and all $\epsilon > 0$ there is $g \in G$ with diameter $\mathrm{diam}(g(I)) < \epsilon$, and
	\item \emph{nonelementary} if there is no $G$-invariant probability measure on $S^1$.
\end{itemize} A theorem of G. Margulis \cite{Margulis2000} (see also \cite{Antonov1984}) states that for any nonelementary subgroup $G \leq \mathrm{Homeo}_+(S^1)$ there exists a $G$-equivariant quotient of $S^1$, still homeomorphic to $S^1$, on which $G$ acts minimally and proximally. Thus, the study of any nonelementary group action on $S^1$ often reduces to the study of a minimal and proximal action (see Subsection  \ref{subsec:groupsCircle} for a precise statement). One can verify that the action of Thompson's group $T$ on $S^1$ is minimal and proximal, and hence nonelementary.

Consider a countable group $G$ acting proximally and minimally on $S^1$, and let $\mu$ be a nondegenerate probability measure on $G$. B. Deroin, V. Kleptsyn and A. Navas \cite{DeroinKleptsynNavas2007} show that there exists a unique $\mu$-stationary probability measure $\nu$ on $S^1$, and that for $\P$-almost every sample path $\mathbf{w} = (w_n)_{n \geq 0} \in G^{\N}$ of the $\mu$-random walk on $G$ there exists a point $\xi(\mathbf{w}) \in S^1$ such that
\(
\lim_{n\to \infty} (w_n)_\ast \nu =\delta_{\xi(\mathbf{w})}
\) in the weak-$\ast$ topology, where $\delta_{\xi(\mathbf{w})}$ denotes the point probability measure at $\xi(\mathbf{w})$. The measure $\nu$ coincides with the distribution of $\xi(\mathbf{w})$ on $S^1$, and therefore the space $(S^1,\nu)$ provides a $\mu$-boundary of $G$.

\subsection{Main results}
We say that an action of a group $G \leq \mathrm{Homeo}_+(S^1)$ on $S^1$ is \emph{topologically nonfree} if there exists $g \in G \setminus \{e_G\}$ whose set of fixed points has nonempty interior. Note that the action of Thompson's group $T$ on $S^1$ is topologically nonfree.

\begin{thmA} \label{thm:entropy}
	Let $G \leq \mathrm{Homeo}_+(S^1)$ be a countable group whose action on $S^1$ is minimal, proximal and topologically nonfree, and let $\mu$ be a finite entropy nondegenerate probability measure on $G$. Then $(S^1,\nu)$ is not the Poisson boundary of $(G, \mu)$.
\end{thmA}

In the particular case of finitely supported nondegenerate probability measures on $T$, Theorem \ref{thm:entropy} answers a question asked by B. Deroin \cite[Item (2) in Section 6]{Deroin2013} and by A. Navas in his ICM survey \cite[Question 19]{Navas2018}. One should contrast Theorem \ref{thm:entropy} with \cite[Theorem 1.1]{Deroin2013}, where B. Deroin proves that for subgroups $G \leq \mathrm{Diff}^2(S^1)$ within a certain class, that contains in particular any cocompact lattice in $\mathrm{PSL}_2(\R)$, the space $(S^1,\nu)$ is the Poisson boundary of $(G, \mu)$ when $\mu$ satisfies suitable moment conditions. These groups fall within a family that is conjectured to be composed only of Fuchsian groups and virtually free groups \cite[Table 1]{AlvarezFilimonovKleptsynMalicetMeninoCotonNavasTriestino2019}, and hence their Poisson boundaries could alternatively be described using their Gromov boundaries; see Subsection \ref{subsec:background}.

Theorem \ref{thm:entropy} is related to the well-known open problem \cite{CannonFloydParry1996} on whether Thompson's group $F$, the group of dyadic piecewise affine homeomorphisms of the interval $[0,1]$, is amenable. Indeed, the action of a countable group $G$ on its Poisson boundary $(\partial_{\mu}G,\nu)$ is amenable \cite[Theorem 5.2]{Zimmer1978}, and hence for $\nu$-almost every $x \in \partial_{\mu}G$ the stabilizer subgroup $G_x \leq G$ is amenable (see \cite[Corollary 5.3.33]{AnantharamanDelarocheRenault2000}). If the circle were the Poisson boundary of $T$ then we would conclude that $F$ is amenable, since for each $x\in S^1$ the stabilizer $T_x \leq T$ contains a copy of $F$. Theorem \ref{thm:entropy} implies that this strategy does not work whenever the probability measure $\mu$ has finite entropy.

The proof of Theorem \ref{thm:entropy} is sketched in Subsection \ref{subsec:background}, and relies on the conditional entropy criterion of V. Kaimanovich \cite{Kaimanovich1985} together with a conditional version of a method used by A. Erschler to show the positivity of asymptotic entropy \cite{Erschler2003}. As is often the case with entropy methods for Poisson boundaries, our proof does not provide explicit bounded $\mu$-harmonic functions that do not arise from the $\mu$-boundary $(S^1, \nu)$ or, equivalently, an explicit $\mu$-boundary that is not a $G$-equivariant quotient of $(S^1, \nu)$. Our second theorem gives precisely this information for a more restricted class of subgroups of $\mathrm{Homeo}_{+}(S^1)$, which still includes Thompson's group $T$. 

We denote by $\mathrm{PAff}_+(S^1)$ the group of piecewise affine orientation-preserving homeomorphisms of $S^1 \cong \R/\Z$ whose derivative has finitely many discontinuity points. Given $g \in \mathrm{PAff}_+(S^1)$ denote by $\mathbf{Br}_g\subseteq S^1$ the (finite) set of discontinuities of its derivative, which we call the \emph{breakpoints} of $g$. For a countable group $G \leq \mathrm{PAff}_+(S^1)$ with a minimal, proximal and topologically nonfree action on $S^1$, we define $\mathbf{Br}=\cup_{g \in G}\mathbf{Br}_g$ the \emph{set of breakpoints of $G$} and show that, with respect to an appropriate action of $G$, there is a $\mu$-stationary probability measure $\widetilde{\nu}$ on $\R^{\mathbf{Br}}$ such that $(\R^\mathbf{Br}, \widetilde{\nu})$ is a $\mu$-boundary of $G$ (see Subsection \ref{subsection: proofs}). We call this space the \emph{breakpoint boundary} of $G$. This construction is analogous to the boundary of Thompson's group $F$ constructed by V. Kaimanovich in \cite{Kaimanovich2017}.

\begin{thmA}\label{thm:lamps} Let $G$ be a countable subgroup of $\mathrm{PAff}_+(S^1)$ whose action on $S^1$ is minimal, proximal, and topologically nonfree, and let $\mu$ be a nondegenerate probability measure on $G$ such that $\sum_{g \in G} \mu(g) \abs{\mathbf{Br}_g}<\infty$. Then the breakpoint boundary $(\R^\mathbf{Br}, \widetilde{\nu})$ is not a $G$-equivariant quotient of $(S^1, \nu)$. In particular, $(S^1, \nu)$ is not the Poisson boundary of~$(G, \mu)$.
\end{thmA}

The class of groups to which Theorem \ref{thm:entropy} applies is larger than that considered in Theorem \ref{thm:lamps}. Indeed, note that $\mathrm{PSL}_2(\Z[\sqrt{2}])$ acts minimally and proximally on $S^1$ through its natural projective action, and it follows from \cite[Theorem 1.4]{Cornulier2021} that it does not embed into $\mathrm{PAff}_+(S^1)$. Consider the group $G$ of all piecewise-$\mathrm{PSL}_2(\Z[\sqrt{2}])$ homeomorphisms of $S^1$ with breakpoints in the set of fixed points of hyperbolic elements in $\mathrm{PSL}_2(\Z[\sqrt{2}])$. Then $G$ is countable, satisfies the hypotheses of Theorem \ref{thm:entropy} and is not conjugate to a subgroup of $\mathrm{PAff}_+(S^1)$ since it contains $\mathrm{PSL}_2(\Z[\sqrt{2}])$.

\subsection{Further background on identification of Poisson boundaries} \label{subsec:background}
Given a probability measure $\mu$ on a countable group $G$, a natural question is to find an explicit model for the Poisson boundary of $(G,\mu)$. That is, to identify the Poisson boundary of the random walk with a concrete $\mu$-boundary expressed in terms of geometric, combinatorial, or algebraic properties of $G$. Currently, the main tools for studying this question are based on entropy. For a probability measure $\mu$ on a group $G$, the \emph{Shannon entropy} of $\mu$ is defined as $H(\mu)\coloneqq -\sum_{g\in G}\mu(g)\log(\mu(g))$. The \emph{asymptotic entropy} of a probability measure $\mu$ is defined by \(h(\mu)\coloneqq\lim_{n\to \infty} H(\mu^{*n})/n.\) This quantity was introduced by A. Avez \cite{Avez1972}, who proved that if $\mu$ is finitely supported and $h(\mu)=0$, then the Poisson boundary of $(G,\mu)$ is trivial. The \emph{entropy criterion} of Derriennic \cite{Derrienic1980} and Kaimanovich-Vershik \cite[Theorem 1.1]{KaimanovichVershik1983} states that if $H(\mu)<\infty$, then $h(\mu)>0$ if and only $(G,\mu)$ has a nontrivial Poisson boundary. This result was later extended by V. Kaimanovich, who proved that a $\mu$-boundary of $G$ is the Poisson boundary if and only if the sequence of \emph{mean conditional} entropies at time $n$ of $\mu$ (see Definition \ref{defn: average conditional entropy}) grows sublinearly \cite[Theorem 2]{Kaimanovich1985} \cite[Theorem 4.6]{Kaimanovich2000} (see also Theorem \ref{thm: conditional entropy alternative formulation} below). It is important to note that these criteria have two main restrictions: the first is that the hypothesis $H(\mu)<\infty$ is crucial and the criteria do not work in a general context for measures with infinite entropy. The second restriction for the entropy criteria is that one needs to identify by other means a $\mu$-boundary of $G$ that serves as a potential candidate for the Poisson boundary. There are families of groups for which one can prove that some random walks on them have a nontrivial Poisson boundary by using the entropy criterion but for which there is no known nontrivial $\mu$-boundary. An example of this is the wreath product $\Z/2\Z\wr \Z^3$, which has nontrivial Poisson boundary for each nondegenerate probability measure with finite entropy \cite[Theorem 3.1]{Erschler2003}.

Next, we recall results on the identification of the Poisson boundary for two families of groups: Gromov-hyperbolic groups and wreath products. For a more comprehensive list of results regarding the identification of Poisson boundaries of random walks on groups, we refer to \cite{Erschler2010} and \cite[Section 3.3.5]{SilvaPhdThesis2024}.

\subsubsection*{Gromov-hyperbolic groups.}
The complete description of nontrivial Poisson boundaries goes back to E. Dynkin and M. Maljutov \cite{DynkinMaljutov1961}, who proved that for any non-abelian free group and any probability measure supported on a free generating set, the corresponding Poisson boundary can be identified with the space of infinite reduced words of the free group endowed with its unique stationary measure. This result was extended to all probability measures with finite support by Y. Derriennic \cite{Derrienic1975}. More generally, A. Ancona \cite{Ancona1987} proved that the Poisson boundary of a non-elementary Gromov-hyperbolic group with respect to a finitely supported measure coincides with its Gromov boundary endowed with the unique stationary measure. V. Kaimanovich later used the conditional entropy criterion to generalize the latter result to hold for any nondegenerate probability measure with finite entropy and finite first logarithmic moment \cite[Theorem 8]{Kaimanovich1994}, \cite[Theorems 7.4 and 7.7]{Kaimanovich2000}. Recently, this description was proved to hold for all nondegenerate measures with finite entropy by K. Chawla, B. Forghani, J. Frisch and G. Tiozzo \cite[Theorem 1.1]{ChawlaForghaniFrischTiozzo2025}. We note that the two latter results were new even in the case of free groups.

\subsubsection*{Wreath products.}
Given countable groups $A$ and $B$, the \emph{wreath product} $A \wr B$ is the semidirect product $\left( \bigoplus_B A \right) \rtimes B$ where $B$ acts on $\bigoplus_B A$ by translations. The main tool used to exhibit a nontrivial $\mu$-boundary of random walks on wreath products is the \emph{stabilization of lamp configurations}: suppose that $\mu$ is a probability measure on $A\wr B$ such that for $\P$-almost every trajectory $\left\{(\varphi_n,X_n)\right\}_{n\ge 0}$ of the $\mu$-random walk on $A\wr B$ and for each $b\in B$, there exists $N\ge 1$ such that for every $n\ge N$, we have $\varphi_{n}(b)=\varphi_N(b)$. From this one can deduce the existence of a $\mu$-stationary probability measure $\nu$ on $\prod_{B}A$ such that $(\prod_{B}A,\nu)$ is a $\mu$-boundary of $A\wr B$. A sufficient condition that guarantees the stabilization of lamp configurations is that $\mu$ has a finite first moment with respect to a word metric on $A\wr B$ and that the projection of $\mu$ to $B$ defines a transient random walk \cite[Section 6]{KaimanovichVershik1983}, \cite[Theorem 3.3]{Kaimanovich1991}, \cite[Theorem 2.9]{KarlssonWoess2007}, \cite[Lemma 1.1]{Erschler2011}. In particular for $B=\Z^d$, $d\ge 1$, if $H(\mu)<\infty$ and $\mu$ satisfies the stabilization of lamp configurations, the space $(\prod_B A,\nu)$ is the Poisson boundary of $(A\wr B,\mu)$ \cite[Theorem 3.6.6]{Kaimanovich2001}, \cite[Theorem 1]{Erschler2011},  \cite[Theorems 1.1 \& 5.1]{LyonsPeres2021}, \cite[Theorem 1.3 \& Corollary 1.4]{FrischSilva2023}. In contrast, there are nondegenerate probability measures on $A \wr \Z^d$, $d\ge 3$, with an infinite first moment and finite entropy such that the lamp configuration does not stabilize, and yet the Poisson boundary is nontrivial \cite[Proposition 1.1]{Kaimanovich1983}, \cite[Section 6]{Erschler2011}, \cite[Section 5]{LyonsPeres2021}. For such probability measures, there are no known constructions of nontrivial $\mu$-boundaries.

\subsubsection*{A natural boundary for Thompson's group $T$.} The $\mu$-boundary $(S^1, \nu)$  can be considered as a ``natural'' candidate for the Poisson boundary of Thompson's group $T$ in the following sense. A faithful action $G \curvearrowright X$ of a group $G$ on a locally compact Hausdorff perfect space $X$ by homeomorphisms is a \emph{Rubin action} if for every open subset $U \subseteq X$ and every $x \in U$, the closure of the orbit of $x$ under the action of the subgroup \(
\{ g \in G \mid \restr{g}{X \setminus U} = \mathrm{id}_{X \setminus U}\}
\) contains a neighborhood of $x$. A theorem by M. Rubin \cite[Corollary 3.5]{Rubin1989} states that there exists a unique Rubin action up to $G$-equivariant homeomorphisms
(see also \cite{BelkElliotMatucci2022} for a recent short proof). One can phrase this result as saying that a group admits at most one ``sufficiently rich'' \emph{microsupported} action on a compact space. Notice that the action of Thompson's group $T$ on $S^1$ is Rubin.

To the best of our knowledge, Thompson's group $T$ is the first example of a group such that the Poisson boundary of a simple random walk is strictly larger than a ``natural'' candidate $\mu$-boundary.
One can compare this result with the recent work of K. Chawla and J. Frisch \cite{ChawlaFrisch2025}, from which it follows that on any non-abelian free group there are probability measures with infinite entropy such that the Poisson boundary is not the Gromov boundary of the free group. 

\subsection{Structure of the proofs}\label{subsection: proofs}

The proof of Theorem \ref{thm:entropy} follows similar steps to a method used by A. Erschler in \cite{Erschler2003} (see also Theorem \ref{thm:erschler}) to show that $h(\mu)>0$ for a probability measure $\mu$ on a group $G$. This method consists of verifying that the sequence $(H(\mu^{\ast n}))_{n \geq 0}$ grows linearly under the following condition: there exist $p, c \in (0,1)$ and $a \in \mathrm{supp}(\mu) \setminus \{e_G\}$ such that for every $n \in \N$, with probability at least $p$ we can choose elements $b_1,\ldots, b_{k+1} \in G$ and times $1 \leq i_1 < i_2 < \cdots < i_k \leq n$ with $k \geq cn$ such that the random walk $w_n$ at step $n$ can be expressed \emph{in a unique way} as \[
w_n = b_1 \epsilon_1 b_2 \cdots b_k \epsilon_k b_{k+1},
\]  where $\epsilon_j \in \{a, e_G\}$ is the increment of the random walk at time $i_j$ for every $1 \leq j \leq k$.

We state a conditional version of Erschler's method in Theorem \ref{thm:relEntropy}, which gives sufficient conditions to ensure that the sequence of conditional entropies at time $n$ with respect to a given $\mu$-boundary $\mathbf{X}$ grows linearly on $n$. Then, thanks to Kaimanovich's conditional entropy criterion and with the hypothesis of $H(\mu)<\infty$, we conclude that the space $\mathbf{X}$ not the Poisson boundary of $(G, \mu)$. The proof of Theorem \ref{thm:entropy} consists on applying Theorem \ref{thm:relEntropy} to the $\mu$-boundary $(S^1,\nu)$, and it follows the next steps.

\begin{itemize}
	\item The main dynamical input that guarantees that we can apply Theorem \ref{thm:relEntropy} in the context of Theorem \ref{thm:entropy} is Proposition \ref{prop:linear}, which asserts that for each sufficiently small interval $I \subset S^1$, there are linearly many intervals in the sequence $I, w_1(I), \ldots, w_n(I)$ such that each of these intervals \emph{dominates} all the ones preceding it. Here, when $I_1, I_2 \subset S^1$ are closed intervals we say that $I_1$ \emph{dominates} $I_2$ if they are either disjoint or if the interior of $I_1$ contains $I_2$.
	\item By the hypotheses on the action of $G$ on $S^1$, Proposition \ref{proposition:small support} ensures that there exists a nontrivial element $a \in G \setminus \{e_G\}$ such that  
	\(
	\overline{\{x \in S^1 : a(x) \neq x\}}\subseteq I.
	\) 
	\item The nondegeneracy assumption on $\mu$ guarantees that, up to replacing $\mu$ with a convolution power $\mu^{\ast s}$ for some $s\ge 1$, we may assume that $a$ belongs to the support of $\mu$. From the above, we obtain $c \in (0,1)$ such that, in expectation, there are at least $k \geq cn$ times $1 \leq i_1 < \cdots < i_k \leq n$ such that $w_{i_j - 1}(I)$ dominates $w_{i_l- 1}(I)$ and $w_{i_j-1}^{-1}w_{i_j} \in \{a, e_G\}$ for all $1 \leq l < j \leq k$.
	\item Given $n \in \N_+$, fix a trajectory $\mathbf{w}$ such that there are at least $k \leq cn$ times $1 \leq i_1 < \cdots < i_k \leq n$ as in the previous step. Write $w_n = b_1 \epsilon_1 b_2 \cdots b_k \epsilon_k b_{k+1}$ where the $\epsilon_1,\ldots, \epsilon_k \in \{a, e_G\}$ are the jumps at times $i_1,\ldots, i_k$. Then, Lemma \ref{lemma:good} says that whenever $\epsilon_1',\ldots, \epsilon_k'$ range over $\{a, e_G\}$, the resulting elements $b_1 \epsilon_1' b_2 \cdots b_k \epsilon_k' b_{k+1}$ are pairwise distinct. From this, we verify the hypotheses of Theorem \ref{thm:relEntropy} and finish the proof.
\end{itemize}

The proof of Theorem \ref{thm:lamps} in Section \ref{sec:harmonic} is different and does not rely on entropy techniques. Given a countable group $G \leq \mathrm{PAff}_+(S^1)$, the set $\mathbf{Br}$ of breakpoints of elements of $G$ is countable, and there is a  map $\cC \colon G \to \R^\mathbf{Br}$ given by $$\cC_g(x) \coloneqq \log\left((g^{-1})'(x^+)\right) - \log\left((g^{-1})'(x^-)\right) \text{ for each }x\in \mathbf{Br},$$ that records the discontinuities of the derivative of $g\in G$. The map $\cC:G\to \R^{\mathbf{Br}}$ is an additive cocycle for the natural action of $G$ on $\R^\mathbf{Br}$ obtained from the action of $G$ on $\mathbf{Br}$ (see Equation \eqref{eq:equivariance}). The breakpoint boundary is defined as follows.
\begin{itemize}
	\item Lemma \ref{lemma:transience} shows that the trajectory in $\mathbf{Br}$ of any breakpoint through the $\mu$-random walk is transient. This relies on a general comparison lemma for Markov operators \cite{BaldiLohouePeyriere1977}.
	\item  For $\P$-almost every $\mathbf{w} = (w_n)_{n \geq 0}$ the configurations $\cC_{w_n}$ converge pointwise to a configuration $\cC_\infty(\mathbf{w}) \in \R^\mathbf{Br}$ as $n\to \infty$. This defines an associated hitting measure $\widetilde{\nu}$ on $\R^\mathbf{Br}$,
	
	\item We define another action of $G$ on $\R^{\mathbf{Br}}$, where for each $g\in G$ and $\mathcal{C}\in \R^{\mathbf{Br}}$ we define $(g\cdot \mathcal{C})(x)=\cC_g(x)+\cC(g^{-1}x)$ for every $x\in \mathbf{Br}.$ With respect to this action, the space $(\R^\mathbf{Br},\widetilde{\nu})$ is a $\mu$-boundary of $G$. 
\end{itemize}
Using the breakpoint boundary $(\R^\mathbf{Br},\widetilde{\nu})$, we find for each $n\ge 1$ a bounded $\mu$-harmonic function $f_n \colon G \to \R$ and an element $a_n \in G$ with small support such that there is a constant $K>0$ with $\abs{f_n(a_n) - f_n(e_G)} > K$ for all $n \in \N$. This cannot occur for bounded harmonic functions obtained from $(S^1,\nu)$.

We remark that the construction outlined in the first two steps of the proof was used by B. Stankov to prove that, under the same moment condition on $\mu$ as in Theorem \ref{thm:lamps}, any subgroup of Monod's nonamenable group $H(\Z)$ of piecewise-$\mathrm{PSL}_2(\Z)$ homeomorphisms of the line is either locally solvable or has nontrivial Poisson boundary \cite[Theorem 1.2]{Stankov2021}.

\subsection{Organization}
In Section \ref{sec:preliminaries} we recall background material on random walks on groups, Poisson boundaries, the conditional entropy criterion, and properties of groups acting on the circle. In Section \ref{sec:entropy} we describe the method of A. Erschler (Theorem \ref{thm:erschler}) used to prove positivity of asymptotic entropy and our conditional version of it (Theorem \ref{thm:relEntropy}). In Section \ref{sec:good} we prove Lemma \ref{lemma:good}, which gives a sufficient condition for us to verify the hypotheses of Theorem \ref{thm:relEntropy}. Next, in Section \ref{sec:mean} we prove two quantitative statements for random walks on $S^1$ that are key to our results: Theorem \ref{thm:meanXi} and Proposition \ref{prop:equidistributionXi}. Afterwards, in Section \ref{sec:linear} we prove Proposition \ref{prop:linear} that guarantees that in expectation there will be linearly many dominating intervals, and use it in Section \ref{sec:proofA} to prove Theorem \ref{thm:entropy}. Finally, in Section \ref{sec:harmonic} we define the breakpoint boundary and prove Theorem \ref{thm:lamps}.

\subsection{Acknowledgements}
Martín Gilabert Vio acknowledges support from the ANR project Gromeov (ANR-19-CE40-0007), and thanks Nicolás Matte Bon for useful conversations. Eduardo Silva is funded by the Deutsche Forschungsgemeinschaft (DFG, German Research Foundation) under Germany's Excellence Strategy EXC 2044 –390685587, Mathematics Münster: Dynamics–Geometry–Structure.
\section{Preliminaries} \label{sec:preliminaries}
\subsection{Random walks on groups and Poisson boundaries} \label{subsec:poisson}
Let $\mu$ be a probability measure on a countable group $G$, and consider the probability measure $\P$ obtained as the push-forward of the Bernoulli measure $\mu^{\mathbb{N}}$ through the map
\begin{equation*}
	\begin{aligned}
		G^{\mathbb{N}}&\to G^{\mathbb{N}}\\
		(g_1,g_2,g_3,\ldots)&\mapsto (w_0,w_1,w_2,w_3,\ldots)\coloneqq (e_G,g_1,g_1g_2,g_1g_2g_3,\ldots).
	\end{aligned}
\end{equation*}
The space $(G^{\mathbb{N}},\P)$ is called the \emph{space of sample paths} or the \emph{space of trajectories} of the $\mu$-random walk. We denote by $\sigma \colon G^\N \to G^\N$ the \emph{shift map} $\sigma((w_n)_{n \geq 0}) = (w_{n+1})_{n \geq 0}$.

The Poisson boundary of the $\mu$-random walk on $G$ was already defined in the introduction as the maximal $\mu$-boundary of $G$. An alternative definition is the following.
\begin{definition}\label{defn: Poisson boundary original def} Let $G$ be a countable group, and let $\mu$ be a probability measure on $G$. Two sample paths $\mathbf{w},\mathbf{w^\prime} \in G^\N$ are said to be equivalent if there exist $i,j \geq 0$ such that $\sigma^i(\mathbf{w}) = \sigma^j(\mathbf{w^\prime})$. Consider the measurable hull associated with this equivalence relation, that is, the $\sigma$-algebra formed by all measurable subsets of the space of trajectories $(G^{\mathbb{N}},\P)$ which are unions of the equivalence classes of $\sim$ up to $\P$-null sets. The associated quotient space is called the \emph{Poisson boundary} of the random walk $(G,\mu)$.
\end{definition}
For further equivalent definitions of the Poisson boundary we refer to the articles \cite{KaimanovichVershik1983}, \cite[Section 1]{Kaimanovich2000} and the references therein. For an overview of the study of Poisson boundaries and random walks on groups we refer to the surveys \cite{Furman2002,Erschler2010,Zheng2022}.

Denote by $H^\infty(G, \mu)$ the space of \emph{bounded harmonic functions}, that is, of bounded functions $f \colon G \to \R$ such that \( f(g) = \sum_{h \in G} f(gh) \mu(h) \) for all $g \in G$. The vector space $H^\infty(G, \mu)$ becomes a Banach space by equipping it with the $l^\infty$ norm. For a probability space $(X, \nu)$ endowed with a measurable action of $G$, we say that $\nu$ is \emph{$\mu$-stationary} if 
\(\nu= \mu \ast \nu \coloneqq \sum_{g \in G} \mu(g) g_\ast \nu. \) In this case, for each $F\in L^{\infty}(X,\nu)$ one can define a $\mu$-harmonic function $f\in H^{\infty}(G,\mu)$ by
\[
f(g) = \int_X F(x) \dd g_\ast \nu=\int_X F(gx) \dd \nu, \text{ for each }g\in G. 
\]  The map $F \in L^{\infty}(X,\nu) \mapsto f \in H^{\infty}(G,\mu)$ is called the \emph{Poisson transform associated with $(X,\nu)$}. When $(X,\nu)$ is the Poisson boundary of $(G,\mu)$, this map is the Poisson transform defined in Equation \eqref{eq: Poisson transform} in the introduction. 

We recall a standard result on the Poisson transform associated to a $\mu$-boundary, which will be used in Section \ref{sec:harmonic} within the proof of Theorem \ref{thm:lamps}. Its proof goes back to the work of H. Furstenberg (see the second paragraph on page 373 of \cite{Furstenberg1963}) and R. Azencott \cite[Proposition I.2]{Azencott1970}, and can be found, for example, in \cite[Theorem 4.4]{Glasner1976}. 
\begin{thm}\label{thm: isometry}
	Let $\mu$ be a probability measure on a countable group $G$ and let $ (X, \nu)$ be a $\mu$-boundary of $G$. Then the Poisson transform associated with $(X, \nu)$ is an isometry between $L^\infty(X, \nu)$ and a closed subspace of $H^\infty(G, \mu)$, and it is surjective if and only if $(X, \nu)$ coincides with the Poisson boundary $(\partial_\mu G, \nu)$.
\end{thm}

\subsection{Conditional probabilities with respect to a $\mu$-boundary}

Let $\mathbf{X}=(X,\nu)$ be a $\mu$-boundary of $G$. Using V. Rokhlin's theory of measurable partitions of Lebesgue spaces, the probability measure $\P$ can be disintegrated with respect to the boundary map $(G^{\mathbb{N}},\P) \to (X, \nu)$. That is, for $\nu$-almost every $\xi \in X$  there is a probability measure $\P^{\xi}$ supported on the fiber of $\xi$ in $G^\N$ such that $\P=\int_{X}\P^{\xi} \dd \nu (\xi)$; see \cite[Section I.7]{Rohlin1967} and \cite[Section 3]{Kaimanovich2000}. These conditional probability measures determine Markov chains on $G$ with transition probabilities
\begin{equation}\label{eq: conditional transition probs}
	\P^{\xi}\left[ w_{n+1}=g\mid w_n=h \right] =\P\left[ w_{n+1}=g\mid w_n=h\right] \frac{\dd g_{*}\nu}{\dd h_{*}\nu}(\xi)=\mu(h^{-1}g)\frac{\dd g_{*}\nu}{\dd h_{*}\nu}(\xi)
\end{equation}
for every $g,h\in G$ and for $\nu$-almost every $\xi\in X$; see Equation $(20)$ in page $462$ of \cite{KaimanovichVershik1983}, or \cite[Theorem 1.3.4]{Kaimanovich2001}.

\begin{lemma}\label{lemma:radonNikodym} Let $\mu$ be a probability measure on a countable group $G$ and consider a $\mu$-boundary $(X,\nu)$ of $G$. Let $J \subseteq X$ be a measurable subset and let $a \in G$ be such that for every $x\in X\backslash J$ we have $a(x)=x$. Then, for every $g\in G$ and $\nu$-almost every $\xi \in X\backslash g(J)$ we have
	\(
	\frac{\dd (ga)_{\ast}\nu}{\dd g_{\ast}\nu}(\xi)=1.
	\) In particular, we have $\P^\xi[w_{n+1} = ga \mid w_n = g] = \mu(a)$ for every $n\ge 1$.
\end{lemma}
\begin{proof}
	Consider an arbitrary measurable subset $A\subseteq X \backslash g(J)$. Then $g^{-1}A\subseteq X\backslash J$, and hence $(ga)^{-1}A=a^{-1}g^{-1}A=g^{-1}A$. From this, we have
	\[ g_{\ast}\nu(A) =\nu(g^{-1}A) =\P\left[ \xi(\mathbf{w})\in g^{-1}A\right] = \P\left[ \xi(\mathbf{w})\in (ga)^{-1}A\right]  =(ga)_{\ast}\nu(A).\]
	Since $A$ was arbitrary, the above implies the first statement of the lemma. The second statement follows from Equation \eqref{eq: conditional transition probs}.
\end{proof}

\subsection{Entropy} \label{subsec:entropy methods}
We consider countable partitions of the space of sample paths $G^{\mathbb{N}}$ in the definition of entropy below. The most important partitions for this are the following.
\begin{definition}\label{def: partition rw at time n}
	For every $k \ge 1$, define the partition $\alpha_k$ of the space of sample paths $G^{\mathbb{N}}$ where two trajectories $\mathbf{w} = (w_n)_{n \geq 0}, \, \mathbf{w^\prime} = (w_n')_{n \geq 0} \in G^{\N}$ belong to the same element of $\alpha_k$ if and only if $w_k=w^\prime_k$.
\end{definition} 
In other words, $\alpha_k$ is the partition given by the element visited by the $\mu$-random walk at time $k$.

The \emph{Shannon entropy} of a countable partition $\rho$ of the space of sample paths $G^{\N}$ with respect to the probability measure $\P$ is defined as
\(
H(\rho)\coloneqq -\sum_{k\ge 1}\P(\rho_k)\log \P(\rho_k).\) Note that $ H(\alpha_1)=H(\mu)$ is the Shannon entropy of the probability measure $\mu$.

\begin{definition}\label{defn: average conditional entropy}
	Let $\rho=\{\rho_k\}_{k\ge 1}$ be a countable partition of the space of sample paths $G^{\mathbb{N}}$ and let $\mathbf{X}=(X,\nu)$ be a $\mu$-boundary of $G$. Consider the disintegration $\P=\int_X \P^{\xi} \dd \nu(\xi)$ with respect to the boundary map $G^{\N}\to X$. We define for $\nu$-almost every $\xi\in X$ the \emph{conditional entropy of $\rho$ given $\xi$} as \(H_{\xi}(\rho)\coloneqq-\sum_{k\ge 1}\P^{\xi}(\rho_k)\log \P^{\xi}(\rho_k).\)
	Following \cite[Section 5.1]{Rohlin1967}, let us define the \emph{mean conditional entropy of $\rho$ over the $\mu$-boundary $\mathbf{X}$} as
	\(	H_{\mathbf{X}}(\rho)\coloneqq \int_{X}H_{\xi}(\rho) \dd \nu(\xi).\)
\end{definition}

We now formulate Kaimanovich's conditional entropy criterion \cite[Theorem 2]{Kaimanovich1985} in terms of the mean conditional entropy, which is our main tool in the proof of Theorem \ref{thm:entropy}.
\begin{thm}\label{thm: conditional entropy alternative formulation}
	Let $G$ be a countable group, and let $\mu$ be a probability measure on $G$ with $H(\mu)<\infty$. Consider a $\mu$-boundary $\mathbf{X}$ of $G$. Then $\mathbf{X}$ is the Poisson boundary of $(G,\mu)$ if and only if
	\begin{equation*}
		h(\mu\mid \mathbf{X})\coloneqq\lim_{n\to \infty}\frac{H_{\mathbf{X}}(\alpha_n)}{n}=0.
	\end{equation*}	
\end{thm}
For the equivalence between Theorem \ref{thm: conditional entropy alternative formulation} and the original formulations of the conditional entropy criterion in \cite[Theorem 2]{Kaimanovich1985} and \cite[Theorem 4.6]{Kaimanovich2000} we refer to the explanation following Theorem 2.4 in \cite{FrischSilva2023}.

In the proof of Theorem \ref{thm:entropy} it will be convenient to modify the step distribution $\mu$ in the following two ways.

\begin{lemma}\label{lemma: lazy boundaries}
	Let $\mu$ be a nondegenerate probability measure on a countable group $G$. 
	Consider a probability measure $\widetilde{\mu}$ equal either to
	\begin{itemize}
		\item $\mu_{\mathrm{lazy}}\coloneqq \frac{1}{2}\mu+\frac{1}{2}\delta_{e_G}$, or to
		\item a convolution  $\mu^{\ast s}$ for some $s \in \N_+$.
	\end{itemize}
	Then $H(\mu)<\infty$ if and only if $H(\widetilde{\mu})<\infty$, and the Poisson boundary of $(G, \mu)$ is $G$-equivariantly measurably isomorphic to the Poisson boundary of $(G, \widetilde{\mu})$.
\end{lemma}
\begin{proof}
	The fact that $H(\mu)<\infty$ if and only if $H(\widetilde{\mu})<\infty$ follows from a direct computation. To see that the Poisson boundaries do not change, note that $\mu_\mathrm{lazy}$ (resp. $\mu^{\ast s}$) can be obtained from $\mu$ by stopping the random walk driven by $\mu$ along the stopping time $\tau = \inf\{k \geq 1 : g_k \neq 0\}$ (resp. $\tau  = s$). The equality of the Poisson boundaries then follows from \cite[Theorem 3.6.1]{ForghaniPhdThesis2015}.
	\end{proof}
	
	\subsection{Groups acting on the circle} \label{subsec:groupsCircle}
	
	We refer to the monographs \cite{Ghys2001, Navas2011} for general references on groups that act on the circle. For an overview on group actions on $1$-manifolds and related topics, we refer to the surveys \cite{Navas2018, Mann2022}.
	
	A continuous surjection $\pi \colon S^1 \to S^1$ is said to have \emph{degree $d \in \N_+$} if any lift $\widetilde{\pi} \colon \R \to \R$ of $\pi$ satisfies $\widetilde{\pi}(x + d) = \widetilde{\pi}(x) + d$ for all $x \in \R$. We say that an action $G \curvearrowright^\phi S^1$ is \emph{semiconjugate} to $G \curvearrowright^\psi S^1$ if there exists a continuous surjection $\pi \colon S^1 \to S^1$ with $\pi \circ \phi(g) = \psi(g) \circ \pi$ for all $g \in G$, and such that $\pi$ is locally nondecreasing and has degree one.
	
	Recall that we call an action $G \curvearrowright S^1$ \emph{minimal} if all orbits are dense in $S^1$. A folklore theorem describing the topological dynamics of an arbitrary group action asserts that unless $G \curvearrowright S^1$ has a finite orbit, for most purposes it suffices to consider minimal actions. For a proof, see \cite[Proposition 5.6]{Ghys2001}.
	
	\begin{thm}
		Consider a group action $G \curvearrowright^\phi S^1$ by orientation-preserving homeomorphisms.
		Then exactly one of the following statements is satisfied.
		\begin{enumerate}
			\item \label{item:class1} There exists a finite orbit.
			\item \label{item:class2} There exists a unique closed minimal set $\Lambda$, which is either $S^1$ or a Cantor set. In the latter case, by collapsing the connected componenents of $S^1 - \Lambda$ we can semiconjugate $\phi$ to a minimal group action $G \curvearrowright S^1$.
		\end{enumerate}
	\end{thm}
	
	Consider a group action $G \curvearrowright^\phi S^1$ by orientation-preserving homeomorphisms. We say that $\phi$ is \emph{locally proximal} if there exists $r > 0$ such that for every interval $I \subset S^1$ with $\mathrm{diam}(I) < r$ and every $\epsilon > 0$ there is $g \in G$ with $\mathrm{diam}(g(I)) < \epsilon$. We say that $\phi$ is \emph{proximal} if the previous is true for all closed intervals $I$ strictly contained in $S^1$.
	
	A complimentary description of the dynamics of a minimal group action on $S^1$ is the following theorem by G. Margulis \cite{Margulis2000}, which also follows from results of V. Antonov \cite{Antonov1984}. We refer to \cite[Section 5.2]{Ghys2001} for a proof and further discussion.
	\begin{thm}
		\label{teo:classification}
		Consider a minimal group action $G \curvearrowright^\phi S^1$ by orientation-preserving homeomorphisms. Then exactly one of the following statements is satisfied.
		\begin{enumerate}
			\item \label{item:class21}the action is conjugated to a minimal action by rotations, or 
			\item \label{item:class22} the action is proximal, or
			\item \label{item:class23} the action is locally proximal and not proximal, and there exists $d \in \N,\, d \geq 2$ and a continuous $d$-to-one covering $\pi \colon S^1 \to S^1$ that intertwines $\phi$ with a proximal action.
		\end{enumerate}
	\end{thm}
	
	As a consequence we have a dichotomy for group actions $G \curvearrowright S^1$: either
	\begin{itemize}
		\item the action preserves a probability measure on $S^1$, and this happens exactly when there is a finite $G$-orbit in $S^1$ or when $G \curvearrowright S^1$ is semiconjugate to a minimal action by rotations, or
		\item there exists $d \in \N_+$ and a continuous surjection $\pi \colon S^1 \to S^1$ that intertwines $G \curvearrowright S^1$ with a minimal and proximal action, such that the fibers $\pi^{-1}(x)$ are of size $d$ for Lebesgue-almost every $x \in S^1$.
	\end{itemize}
	
	Denote by $\mathrm{Leb}$ the Lebesgue measure on $S^1 \cong \R/\Z$. The next theorem, proved in \cite[Appendice]{DeroinKleptsynNavas2007}, shows how $\mu$-boundaries arise in $S^1$ for random walks on groups acting proximally. Its proof may be found in \cite[Section 2.3.2]{Navas2011}.
	
	\begin{thm}
		\label{teo:dkn}
		Consider a group action $G \curvearrowright S^1$ by orientation-preserving homeomorphisms with no invariant probability measure on $S^1$, and let $\mu$ be a nondegenerate probability measure on $G$.
		
		\begin{enumerate}
			\item\label{it:dkn1} There exists a unique $\mu$-stationary probability measure $\nu$ on $S^1$, which is atomless and is supported on the minimal set of $G$.
			\item\label{it:dkn2} If the action of $G$ on $S^1$ is proximal, there exists a random variable $\mathbf{w} \in G^\N \mapsto \xi(\mathbf{w}) \in S^1$ such that for $\P$-almost every $\mathbf{w} = (w_n)_{n \geq 0}$ we have
			\[(w_n)_\ast\mathrm{Leb} \xrightarrow[n\to \infty]{} \delta_{\xi(\mathbf{w})} \] in the weak-$\ast$ topology.
		\end{enumerate}
		In particular, $\lim_{n \to \infty} (w_n)_\ast \nu = \delta_{\xi(\mathbf{w})}$ holds $\P$-almost surely, and hence the circle $(S^1,\nu)$ is a $\mu$-boundary of $G$.
	\end{thm}
	
	Given $g \in \mathrm{Homeo}_+(S^1)$, denote by $\mathrm{supp}(g)$ the closure of the set $\{x \in S^1 : g(x) \neq x\}$. Recall that an action $G \curvearrowright S^1$ is \emph{topologically nonfree} if there exists an element $g \in G$ such that $\mathrm{supp}(g)$ is nonempty and is not all of $S^1$.
	
	The proof of Theorem \ref{teo:classification} shows that, whenever $G$ is a group acting minimally and proximally on $S^1$, any open subset of $S^1$ can be contracted into any nonempty open subset of $S^1$ under the action of $G$. For the convenience of the reader, we present here a (probabilistic) proof of this result.
	\begin{proposition} \label{proposition:small support}
		Consider a minimal and proximal group action $G \curvearrowright S^1$ by orientation-preserving homeomorphisms. Then, for any pair of nonempty closed intervals $I, J \subsetneqq S^1$ with nonempty interior there exists $g \in G$ such that $g(I) \subseteq J$. If the action $G \curvearrowright S^1$ is furthermore topologically nonfree, then for every nontrivial interval $J \subseteq S^1$ there exists $a \in G \setminus \{e_G\}$ such that $\mathrm{supp}(a) \subseteq J$.
	\end{proposition}
	\begin{proof}
		Let $\mu$ be a probability measure on $G$ with $\supp{\mu}=G$ and let $I,J$ be nonempty closed proper intervals of $S^1$ with nonempty interior. Item \eqref{it:dkn2} of Theorem \ref{teo:dkn} states that for $\P$-almost every trajectory $\mathbf{w}\in G^{\mathbb{N}}$ and for any closed interval $K \subseteq S^1$ that does not contain $\xi(\mathbf{w})$, we have $\lim_{n \to \infty} \mathrm{diam}(w_n^{-1}(K)) = 0$. The distribution $\nu$ of $\xi(\mathbf{w})$ has $\supp{\nu}=S^1$, and therefore with positive probability $\xi(\mathbf{w})\notin I$. From this, we obtain $\lim_{n \to \infty} \mathrm{diam}(w_n^{-1}(I)) = 0$. However, since $G \curvearrowright S^1$ is minimal and $\mu$ is nondegenerate, we have that $\P$-almost surely for any $x \in S^1$ the orbit $\{w_n^{-1}(x)\}_{n \geq 0}$ is dense in $S^1$ \cite[Theorem 3.3]{Furman2002}. In particular, when $x$ is the left endpoint of $I$, the above implies that $\mathbf{w} \in G^\N$ and $n \in \N$ such that $w_n^{-1}(I) \subset J$. This shows the first statement of the proposition.
		
		Now let us suppose that the action of $G$ on $S^1$ is topologically nonfree then there is $b \in G \setminus \{e_G\}$ such that $\mathrm{supp}(b)\neq S^1$. Thus, we may choose a closed interval $I \subsetneqq S^1$ that contains $\mathrm{supp}(b)$. Using the first statement of the proposition, we can find $g \in G$ such that $g(I) \subseteq J$. Then, the element $a = gbg^{-1} \in G \setminus \{e_G\}$ satisfies $\mathrm{supp}(a) \subset J$. This shows the second statement of the proposition.
	\end{proof}
	
\section{Lower bounds on entropy} \label{sec:entropy}
In this section we prove a general criterion to show that mean conditional asymptotic entropy is positive. It consists of a conditional version of the method used by A. Erschler \cite{Erschler2003} to show that the asymptotic entropy $h(\mu)$ of a random walk is positive. 

\subsection{Estimating asymptotic entropy}
We start by recalling the method of \cite{Erschler2003} to estimate the asymptotic entropy of a random walk on a group. We do this exclusively for expositional purposes since this result will not be used in the proofs of our theorems. Nonetheless, we believe that understanding the statement of Theorem \cite[Theorem 2.1]{Erschler2003} (see Theorem \ref{thm:erschler} below) is useful for the understanding of its conditional version that we state below (Theorem \ref{thm:relEntropy}), and which is key for the proof of Theorem \ref{thm:entropy}.

Let us first recall some definitions from \cite{Erschler2003} for a probability measure $\mu$ on a countable group $G$. A \emph{collection of length $n \in \N$} is a tuple $q = (\mathbf{x}, i_1, \ldots, i_k)$ where $0 < i_1 < \cdots < i_k \leq n$ are integers and
\[
\mathbf{x} = (x_1, x_2,\ldots, x_{i_1-1}, x_{i_1 + 1},\ldots, x_{i_2 - 1}, x_{i_2 + 1}, \ldots, x_{i_k-1}, x_{i_k + 1}, \ldots, x_n)
\] is an $(n-k)$-tuple of elements of $\mathrm{supp}(\mu)$. Notice that we index the elements of $\mathbf{x}$ by integers in $\{1,\ldots, n\} \setminus \{i_1,\ldots, i_k\}$. 

For each $a,b \in \mathrm{supp}(\mu)$ and $q = (\mathbf{x}, i_1, \ldots, i_k)$ a collection of length $n$, let us define $T^{a,b}(q)$ to be the \emph{set of trajectories} $\mathbf{y} = (y_1, \ldots, y_n) \in  G^n$ such that, after setting $y_0 = 1$, $i_0 = 0$, we have
\begin{itemize}
	\item for all $l \in \{1,\ldots, n\} \setminus \{i_1,\ldots, i_k\}$ we have $y_l = y_{l-1}x_l$, and
	\item for all $l \in \{i_1, \ldots, i_l\}$, we have $y_l = y_{l-1} a$ or $y_l = y_{l-1} b$. 
\end{itemize} Thus, $T^{a,b}(q)$ is a set of $2^k$ trajectories of the $\mu$-random walk on $G$ up to time $n$. We say that $T^{a,b}(q)$ is \emph{satisfactory} if all trajectories in $T^{a,b}(q)$ arrive at different elements of $G$ at time $n$. That is, if $(y_1, \ldots ,y_n), (y'_1, \ldots ,y'_n)$ are distinct trajectories in $T^{a,b}(q)$, then $y_n \neq y'_n$. 

For a trajectory $\mathbf{y} = (y_1, \ldots, y_n)$, we define the \emph{jumps} of $\mathbf{y}$ as $g_j = y_{j-1}^{-1} y_j \in \mathrm{supp}(\mu)$ for all $1 \leq j \leq n$, and we denote by 
\[
[\mathbf{y}] = \{\mathbf{w} \in G^\N \mid w_j = y_j \text{ for all }1 \leq j \leq n\}
\] the cylinder defined by $\mathbf{y}$.

\begin{thm}[\cite{Erschler2003}] \label{thm:erschler}
	Let $\mu$ be a probability measure on a countable group $G$ with $H(\mu)<\infty$. Suppose that there exist $p,c>0$ such that for each $n \in \N_+$ there is a set $A_n$ of collections of length $n$ that verify the following conditions.
	\begin{enumerate}
		\item \label{item:entropy1} For each $q=(\mathbf{x},i_1,\ldots, i_k)\in A_n$ we have $k\ge cn$.
		\item \label{item:entropy3} For each $q\in A_n$ the set of trajectories $T^{a,e_G}(q)$ is satisfactory.
		\item \label{item:entropy4} For each $q_1,q_2\in A_{n}$ with $q_1\neq q_2$ we have $[\mathbf{y}_1] \cap [\mathbf{y}_2]= \emptyset$ whenever $\mathbf{y}_1 \in T^{a,e_G}(q_1)$ and $\mathbf{y}_2 \in T^{a,e_G}(q_2)$.
		\item \label{item:entropy5} We have
		\[
		\P\left[ \bigcup_{q\in A_n} \bigcup_{\mathbf{y} \in T^{a,e_G}(q)} [\mathbf{y}] \right]\ge p.
		\]
	\end{enumerate}
	Then  $h(\mu)>0$.
\end{thm}

To gain some intuition, here is a vague rephrasing of the assumptions. With positive probability, the random walk trajectory is uniquely assigned a linear number of distinguished times along the trajectory together with fixed steps for all other time instants. At any of these distinguished times, the choice between doing an increment of $e_G$ or of $a$ will lead the trajectory to different endpoints, regardless of which choices are made at later distinguished instants.

\subsection{Estimating conditional asymptotic entropy}

We have the following conditional version of Theorem \ref{thm:erschler}, which has a similar structure to Theorem \ref{thm:erschler} but where an additional assumption is needed to handle the transition probabilities conditional to a boundary point.

\begin{thm} \label{thm:relEntropy}
	Let $\mu$ be a probability measure on a countable group $G$ with $H(\mu)<\infty$ and let $\mathbf{X} = (X,\nu)$ be a $\mu$-boundary of $G$. Consider an element $a \in G \setminus \{e_G\}$ and a measurable subset $J \subset X$ such that $a(x)=x$ for each $x\in X\backslash J$. Suppose that there exist $p,c > 0$, for each $n \in \N_+$ a set $\Xi_n \subseteq X$ of measure $\nu(\Xi_n) \geq p$, and for $\nu$-almost every $\xi \in \Xi_n$ a set $A_{n,\xi}$ of collections of length $n$ that verify the following conditions.
	\begin{enumerate}
		\item \label{item:rentropy1} For each $q=(\mathbf{x},i_1,\ldots, i_k)\in A_{n,\xi}$ we have $k\geq cn$.
		\item \label{item:rentropy3} For each $q\in A_{n,\xi}$ the set of trajectories $T^{a,e_G}(q)$ is satisfactory.
		\item \label{item:rentropy4} For each $q_1,q_2\in A_{n,\xi}$ with $q_1\neq q_2$ we have $[\mathbf{y}_1] \cap [\mathbf{y}_2]= \emptyset$ whenever $\mathbf{y}_1 \in T^{a,e_G}(q_1)$ and $\mathbf{y}_2 \in T^{a,e_G}(q_2)$.
		\item \label{item:rentropy5} We have \[
		\P^{\xi}\left[ \bigcup_{q\in A_{n,\xi}} \bigcup_{\mathbf{y} \in T^{a,e_G}(q)} [\mathbf{y}]\right]\geq p.
		\]
		\item \label{item:rentropy2} For each $q= (\mathbf{x},i_1,\ldots, i_k)\in A_{n,\xi}$ and any $(y_1, \cdots, y_n) \in T^{a,e_G}(q)$ we have  \[    	y_{i_r - 1}^{-1}(\xi) \not \in J \text{ for all } 1 \leq r \leq k. \]
	\end{enumerate}
	Then the asymptotic mean conditional entropy $h(\mu \mid X) = \lim_{n \to \infty} H_\mathbf{X}(\alpha_n)/n$ is positive.
\end{thm}
\begin{proof}
	Note that Condition \eqref{item:rentropy4} implies that all the sets $A_{n, \xi}$ are nonempty, and hence the definition of $T^{a,e_G}(q)$ for a collection $q$ and Condition \eqref{item:rentropy1} imply that $a$ and $e_G$ belong to the support of $\mu$.
	
	Let us fix the notation that we will use in the rest of the proof. Given a collection $q$, denote by $Q^{a,e_G}(q) \subset G^\N$ the union $\bigcup_{\mathbf{y} \in T^{a,e_G}(q)} [\mathbf{y}]$. Fix $n \in \N$, and consider $\xi \in X$, a countable measurable partition $\eta$ of $G^\N$ and $q$ a collection of length $n$ such that $\P^\xi[Q^{a,e_G}(q)] > 0$.  Denote by $H^\xi(\eta, q)$ the entropy of the partition
	\(\{P \cap Q^{a,e_G}(q) \mid P \in \eta\} \) of $Q^{a,e_G}(q)$ with respect to the normalized probability measure $\P^\xi/\P^\xi\left[Q^{a,e_G}(q)\right]$ restricted to $Q^{a,e_G}(q)$. If $\P^\xi\left[Q^{a,e_G}(q)\right] = 0$, we set $H^\xi(\eta, q) = 0$.
	
	Let $\alpha_n$ be the partition of $G^\N$ where two trajectories belong to the same atom of $\alpha_n$ if and only if they hit the same element of the group at time $n$. Then we have
	\begin{equation} \label{eq:relEntropyBound}
		H_{\mathbf{X}}(\alpha)= \int_{X} H^\xi(\alpha_n) \dd \nu(\xi)\ge \int_X \sum_{q \in A_{n,\xi}}H^\xi(\alpha_n, q) \P^\xi\left[Q^{a,e_G}(q)\right] \dd \nu(\xi),
	\end{equation} where in the last inequality we used the fact that the $Q^{a,e_G}(q)$ are disjoint, thanks to Condition \eqref{item:rentropy4}.
	
	By Condition \eqref{item:rentropy3}, the partition $\{P \cap Q^{a,e_G}(q) \mid P \in \alpha_n\}$ coincides with the partition
	\[
	\{\, [\mathbf{y}] : \mathbf{y} \in T^{a,e_G}(q)\}.
	\] Hence, by Lemma \ref{lemma:radonNikodym} and Condition \eqref{item:rentropy2}, we obtain that for any $\mathbf{y} \in T^{a,e_G}(q)$ with increments $g_1, \ldots, g_n$ we have
	\begin{align*}
		\P^\xi\left[ [\mathbf{y}] \right]& = \prod_{j = 1}^n \P^\xi\left[ w_j = g_1 \cdots g_{j-1} g_j \mid w_{j-1}= g_1 \cdots g_{j-1}\right]\\
		& = \prod_{r = 1}^k \mu(g_{i_r}) \prod_{\substack{j = 1 \\ j \not \in \{i_s\}_s}}^n \P^\xi\left[ w_j = g_1 \cdots g_{j-1} g_j \mid w_{j-1}= g_1 \cdots g_{j-1}
		\right] \\
		& = \prod_{r = 1}^k \frac{\mu(g_{i_r})}{\mu(a)+\mu(e_G)} 
		\prod_{r = 1}^k \P^\xi\left[ w_{i_r} w_{i_r-1}^{-1} \in \{a, e_G\} \mid w_{i_r-1} \right]
		\prod_{\substack{j = 1 \\ j \not \in \{i_s\}_s}}^n \P^\xi\left[ w_{j} w_{j-1}^{-1} = x_j \mid w_{j-1} \right]
		\\ & = \prod_{r = 1}^k \frac{\mu(g_{i_r})}{\mu(a)+\mu(e_G)} \P^\xi\left[ Q^{a,e_G}(q)\right].
	\end{align*} We deduce that
	\[
	\frac{\P^\xi[[\mathbf{y}]]}{\P^\xi\left[ Q^{a,e_G}(q)\right] }= \frac{\mu(a)^{A(\mathbf{y})}\mu(e_G)^{k - A(\mathbf{y})}}{(\mu(a) + \mu(e_G))^k},
	\] where $A(\mathbf{y}) \in \N$ is the number of times that $a$ appears in the sequence $(g_{i_r})_{r = 1}^k$. Denote by $\rho$ the Bernoulli measure on $\{a, e_G\}$ giving weight $\frac{\mu(a)}{\mu(e_G)+\mu(a)}$ to $a$ and $\frac{\mu(e_G)}{\mu(e_G)+\mu(a)}$ to $e_G$. We conclude that $H^\xi(\alpha_n, q) = H(\rho^{k(q)}) = k(q)H(\rho)$. Since $\mu(a)$ and $\mu(e_G)$ are positive, the quantity $H(\rho)$ is also positive.
	
	We have thus from Condition \eqref{item:rentropy1} that $H^\xi(\alpha_n, q) = k(q) H(\rho) \geq cn H(\rho)$, and therefore
	\begin{align*}
		\int_{\Xi_n} \sum_{q \in A_{n,\xi}}H^\xi(\alpha_n, q) \P^\xi\left[Q^{a,e_G}(q)\right] \dd \nu(\xi) & \geq cn H(\rho)  \int_{\Xi_n} \P^\xi\left[ \bigcup_{q \in A_{n,\xi}} Q^{a,e_G}(q) \right] \dd \nu(\xi) \\
		& \geq cn H(\rho) p^2.
	\end{align*} Finally, from Equation \eqref{eq:relEntropyBound} we deduce that $H_{\mathbf{X}}(\alpha) \geq cH(\rho)p^2 n$ for any $n \in \N$. This shows that $h(\mu\mid \mathbf{X})\ge cH(\rho)p^2>0$ and finishes the proof.
\end{proof}

\section{Good collections} \label{sec:good}
Let $\mu$ be a nondegenerate probability measure on a countable subgroup $G$ of $\mathrm{Homeo}_+(S^1)$ acting minimally and proximally on $S^1$. Let $a \in G \setminus \{e_G\}$ be an element such that $S^1 \setminus \mathrm{supp}(a)$ has nonempty interior, and let $J \subsetneqq S^1$ be the smallest closed interval containing $\mathrm{supp}(a)$.

\begin{definition} \label{def: domination of intervals}
	Given two closed intervals $I_1, I_2 \subset S^1$, we say that $I_1$ \emph{dominates} $I_2$ if they are disjoint or if the interior of $I_1$ contains $I_2$.
\end{definition}

\begin{lemma} \label{lemma:good}
	Fix a collection $q = (\mathbf{x}, i_1, \ldots, i_k)$ and consider two trajectories of length $n$ given by $(y_1, \ldots, y_n), (\widetilde{y}_1, \ldots, \widetilde{y}_n) \in T^{a,e_G}(q)$.
	
	Then, for each $1 \leq r \leq k$ the following statements are equivalent.
	\begin{enumerate}
		\item \label{item:good1} For all $0 \leq l < i_r - 1$, the interval $y_{i_r - 1}(J)$ dominates $y_l(J)$.
		\item \label{item:good2} For all $0 \leq l < i_r - 1$, the interval $\widetilde{y}_{i_r - 1} (J)$ dominates $\widetilde{y}_l(J)$.
	\end{enumerate}
	Whenever these conditions are met, we have $y_{i_r -1}(J) = \widetilde{y}_{i_r - 1}(J)$ for all $0 \leq r \leq k$ and that $q$ is satisfactory.
\end{lemma}

\begin{proof}
	By symmetry, to prove the first statement it suffices to see that condition \eqref{item:good1} implies condition \eqref{item:good2}. Denote by $g_1, \ldots, g_n$ and $\widetilde{g}_1, \ldots, \widetilde{g}_n$ the jumps of $(y_1, \ldots, y_n)$ and $(\widetilde{y}_1, \ldots, \widetilde{y}_n)$ respectively. We have $i_{r-1} < i_r - 1$ for every $1 \leq r \leq k$: indeed, if $i_{r-1} = i_r - 1$ then $y_{i_r - 2}^{-1} y_{i_r}(J) = g_{i_{r-1}} g_{i_r}(J) = J$ so $y_{i_r}(J) = y_{i_r - 2}(J)$, contradicting condition \eqref{item:good1}.
	
	Fix $1 \leq r \leq k$. We will prove by backwards induction on $l = i_r - 2, i_r - 3, \ldots, 0$ that
	\begin{equation} \label{eq:goodInduction} \tag{$H_{\lind}$}
		\widetilde{y}_l^{-1} \widetilde{y}_{i_r - 1}(J) \text{ dominates } J \quad \text{and} \quad \widetilde{y}_l^{-1} \widetilde{y}_{i_r - 1}(J) = y_l^{-1} y_{i_r - 1}(J).
	\end{equation}
	For the base case $l = i_r - 2$, notice that since $i_r - 1 \neq i_{r-1}$, we have $\widetilde{y}_l^{-1} \widetilde{y}_{i_r - 1} = g_{i_r - 1} = y_l^{-1} y_{i_r - 1}$, showing{\renewcommand\lind{i_r - 2} \eqref{eq:goodInduction}}.
	
	Now take $0 \leq l < i_r - 2$ and assume {\renewcommand\lind{l+1}\eqref{eq:goodInduction}}. 
	
	If $l = i_k$ for some $0 \leq k \leq r$, then {\renewcommand\lind{l+1}\eqref{eq:goodInduction}} implies that $\widetilde{y}_{l+1}^{-1} y_{i_r - 1}(J)$ is either disjoint or contains $J$, the support of $\widetilde{g}_l$. Hence 
	\begin{equation} \label{eq:induction1a}
		\widetilde{y}_{l}^{-1} y_{i_r - 1}(J) = \widetilde{g}_l \widetilde{y}_{l+1}^{-1} y_{i_r - 1}(J) = \widetilde{y}_{l+1}^{-1} y_{i_r - 1}(J),
	\end{equation} and we deduce that
	\begin{equation} \label{eq:induction1b}
		\widetilde{y}_{l}^{-1} \widetilde{y}_{i_r - 1}(J) = \widetilde{g}_l \widetilde{y}_{l+1}^{-1} \widetilde{y}_{i_r - 1}(J)  \text{ dominates } J = \widetilde{g}_l(J)
	\end{equation} since {\renewcommand\lind{l+1}\eqref{eq:goodInduction}} holds. Moreover, by condition \eqref{item:good1} the interval $y_{l+1}^{-1} y_{i_r - 1}(J)$ is either disjoint or contains $J$, the support of $g_l$, so
	\[
	y_l^{-1} y_{i_r - 1}(J) = g_l y_{l+1}^{-1} y_{i_r - 1}(J) = y_{l+1}^{-1} y_{i_r - 1}(J)
	\] too, which together with Equation \eqref{eq:induction1a} gives
	\[
	y_l^{-1} y_{i_r - 1}(J) = y_{l+1}^{-1} y_{i_r - 1}(J) = \widetilde{y}_{l+1}^{-1} y_{i_r - 1}(J) = \widetilde{y}_k^{-1} y_{i_r - 1}(J).
	\] The previous equation and Equation \eqref{eq:induction1b} give {\renewcommand\lind{l}\eqref{eq:goodInduction}}.
	
	If $l \neq i_k$ for all $0 \leq k \leq r$ instead, then
	\[
	\widetilde{y}_{l}^{-1} \widetilde{y}_{i_r - 1}(J) =  g_l \widetilde{y}_{l+1}^{-1} \widetilde{y}_{i_r - 1}(J) = g_l y_{l + 1}^{-1}y_{i_r - 1}(J) = y_l^{-1} y_{i_r - 1} (J)
	\]
	and hence
	\(
	\widetilde{y}_{l}^{-1} \widetilde{y}_{i_r - 1} (J) = y_{l}^{-1} y_{i_r - 1}(J) \text{ dominates } J.
	\) This shows \eqref{eq:goodInduction}, finishing the induction and the proof of condition \eqref{item:good2}.
	
	To show the second statement, it remains to show that the collection $q$ is satisfactory if there exists a collection in $T^{a,e_G}(q)$ satisfying condition \eqref{item:good1}. To show this, take two distinct trajectories $(y_1, \ldots, y_n)$ and $(\widetilde{y}_1,\ldots, \widetilde{y}_n)$ in $T^{a,e_G}(q)$ with jumps $g_{i_1},\ldots, g_{i_r}$ and $\widetilde{g}_{i_1},\ldots, \widetilde{g}_{i_r} \in \{a, e_G\}$ respectively at times $i_1, \ldots, i_k$. For $1 \leq r \leq k$ write $e_r = \widetilde{g}_{i_r} g_{i_r}^{-1} \in G$. Denote by $b_1,\ldots, b_k \in G$ the blocks of $q$ between the times $i_1, \ldots, i_k$, so $y_n = b_1g_{i_1}b_2 g_{i_2}\cdots b_k g_{i_k} b_{k+1}$ and $\widetilde{y}_n = b_1 \widetilde{g}_{i_1} b_2 \widetilde{g}_{i_2} \cdots b_k \widetilde{g}_{i_k} b_{k+1}$.
	
	Write
	\begin{align*}
		\widetilde{y}_n y_n^{-1} & = b_1 \widetilde{g}_{i_1} \cdots \widetilde{g}_{i_{k-1}} b_k e_k b_k^{-1} g_{i_{k-1}}^{-1} \cdots g_{i_1}^{-1} b_1^{-1} \\
		& = b_1 \widetilde{g}_{i_1} \cdots \widetilde{g}_{i_{k-2}} b_{k-1} e_{k-1} e_k^{g_{i_{k-1}} b_k} b_{k-1}^{-1} g_{i_{k-2}}^{-1} \cdots g_{i_1}^{-1} b_1^{-1} \\
		& = b_1 \widetilde{g}_{i_1} \cdots \widetilde{g}_{i_{k-2}} b_{k-1} e_{k-1}  b_{k-1}^{-1} g_{i_{k-2}}^{-1} \cdots g_{i_1}^{-1} b_1^{-1} e_k^{b_1 g_{i_1} \cdots g_{i_{k-1}} b_k}
	\end{align*} where we have used the notation $u^v = vuv^{-1}$. By iterating the previous calculation we arrive at
	\[
	\widetilde{y}_n y_n^{-1} = e_1^{b_1} e_2^{b_1 g_{i_1} b_2 }\cdots e_{k-1}^{b_1 g_{i_1} \cdots g_{i_{k-2}} b_{k-1}} e_k^{b_1 g_{i_1} \cdots g_{i_{k-1}} b_k} = e_1^{y_{i_1 - 1}} e_2^{y_{i_2 - 1}} \cdots e_{k-1}^{y_{i_{k-1} - 1}} e_k^{y_{i_k - 1}}.
	\]
	
	\begin{figure}[h!]
		\centering
		\begin{tikzpicture}[scale=1]
			\draw (0,0) circle (2.2);
			
			\draw[thick, blue] ([shift=(22.5:2)]0,0) arc (22.5:130:2) node[black, above right, pos=.5] {$y_{i_r - 1}(J)$};
			\draw[thick, red] ([shift=(40:1.8)]0,0) arc (40:90:1.8) node[black, above right, pos=.5] {};
			\draw[thick, red] ([shift=(100:1.8)]0,0) arc (100:120:1.8) node[black, above right, pos=.5] {};
			\draw[thick, red] ([shift=(50:1.6)]0,0) arc (50:60:1.6) node[black, above right, pos=.5] {};
			
			\draw[thick, red] ([shift=(150:2)]0,0) arc (150:200:2) node[black, above right, pos=.5] {};
			\draw[thick, red] ([shift=(170:1.8)]0,0) arc (170:190:1.8) node[black, above right, pos=.5] {};
		\end{tikzpicture}
		\caption{All red intervals $y_{i_{r'} - 1}(J)$ are dominated by the blue interval $y_{i_r - 1}(J)$.}
		\label{fig:dominatingIntervals}
	\end{figure}
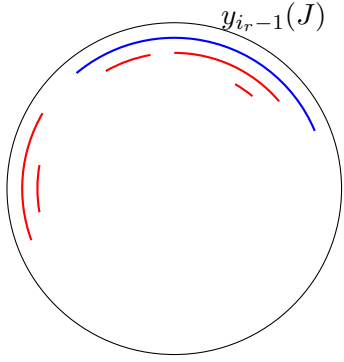

	Notice that the support of $e_r^{y_{i_r - 1}}$ is $y_{i_r - 1}(\mathrm{supp}(a))$. Choose $1 \leq r \leq k$ such that $e_r \neq 1$ and for all $1 \leq r' \leq k$ with $e_{r'} \neq 1$ either $y_{i_{r'} - 1}(J)$ is strictly contained in $y_{i_r - 1}(J)$ or is disjoint from $y_{i_r - 1}(J)$. Then $\widetilde{y}_n y_n^{-1}$ coincides in a neighborhood of $\partial y_{i_r - 1}(J)$ with $e_r^{y_{i_r - 1}}$ (see Figure \ref{fig:dominatingIntervals}). This neighborhood must intersect $y_{i_r - 1}(\mathrm{supp}(a))$ because $J$ is the smallest subinterval containing $\mathrm{supp}(a)$, and hence $\widetilde{y}_n y_n^{-1}$ is nontrivial near $\partial y_{i_r - 1}(J)$. We conclude that $q$ is satisfactory.
\end{proof}

\begin{definition} Given $\xi \in S^1$, we say that $q = (\mathbf{x}, i_1,\ldots, i_k)$ is \emph{$\xi$-good} if
	\begin{itemize}
		\item for all trajectories $(y_1,\ldots, y_n) \in T^{a,e_G}(q)$, every $ 1 \leq r \leq k$, we have
		\[
		y_{i_r - 1}(J) \text{ dominates } y_l(J) \text{ for all } 0 \leq l < i_r - 1 \quad \text{ and } \quad  y_{i_r}^{-1}(\xi) \not \in J, \text{ and}
		\]
		\item the set of indices $i_1, \ldots, i_k$ is maximal among the subsets of $\{1, \ldots, n\}$ satisfying the property above.
	\end{itemize}
\end{definition}
Equivalently, by Lemma \ref{lemma:good} the collection $q$ is $\xi$-good if there exists at least one trajectory in $T^{a,e_G}(q)$ that verifies the previous conditions.

\begin{lemma} \label{lemma:partition}
	For every $n \in \N$ and $\xi \in S^1$ the set of trajectories is partitioned as 
	\[
	G^\N = \bigsqcup_{\substack{q \text{ a $\xi$-good}\\ \text{collection of length }n }} \bigsqcup_{\mathbf{y} \in T^{a,e_G}(q)}[\mathbf{y}].
	\]
\end{lemma}

\begin{proof}
	Given any trajectory $\mathbf{w} = (w_n)_{n \geq 0}$ there is exactly one $\xi$-good collection $q$ such that $\mathbf{w}$ is contained in $\bigcup_{\mathbf{y} \in T^{a,e_G}(q)}[\mathbf{y}]$, which is defined by setting $q = (\mathbf{x}, i_1, \ldots, i_k)$ where the $i_r$ are exactly the indices such that 
	\begin{itemize}
		\item for all $0 \leq l < i_r - 1$ the interval $w_{i_r - 1}(J)$ dominates $w_l(J)$ and
		\item $w_{i_r}^{-1}(\xi) \not \in J$,
	\end{itemize} and $\mathbf{x}$ is obtained from $\mathbf{w}$ by keeping the coordinates in $\{1,\ldots, n\} \setminus \{i_1,\ldots, i_k\}$.
\end{proof}

Note that whenever $A_{n,\xi}$ is composed of $\xi$-good collections, Lemma \ref{lemma:good} implies that conditions \eqref{item:rentropy3} and \eqref{item:rentropy2} in Theorem \ref{thm:relEntropy} are immediately satisfied. Moreover, the previous lemma ensures that condition \eqref{item:rentropy4} in the theorem is also verified. We record this as a proposition, which we will use to verify some of the hypotheses of Theorem \ref{thm:relEntropy}.

\begin{proposition} \label{proposition:rconditions}
	Let $\mu$ be a nondegenerate probability measure with finite entropy on a countable subgroup $G$ of $\mathrm{Homeo}_+(S^1)$ acting minimally and proximally on $S^1$. Consider the $\mu$-boundary $(S^1,\nu)$ of $G$. Let $a \in G \setminus \{e_G\}$ be an element such that both $\mathrm{supp}(a)$ and $S^1\setminus \mathrm{supp}(a)$ have nonempty interior, and let $J \subsetneqq S^1$ be the smallest closed interval containing $\mathrm{supp}(a)$. Let $p> 0$ and consider for each $n \in \N_+$ a subset $\Xi_n \subseteq S^1$ with $\nu(\Xi_n) \geq p$, and for $\nu$-almost every $\xi \in \Xi_n$ a set $A_{n,\xi}$ of collections of length $n$.  If for every $n \in \N$ and $\xi \in \Xi_n$ the set $A_{n,\xi}$ is composed of $\xi$-good collections, then Conditions \eqref{item:rentropy3}, \eqref{item:rentropy4} and \eqref{item:rentropy2} in Theorem \ref{thm:relEntropy} are satisfied.
\end{proposition}
\section{Exponential contraction in mean} \label{sec:mean}

The purpose of this section is to prove Theorem \ref{cor:meanBoundary} and Proposition \ref{prop:equidistributionXi} below, which are the remaining statements on random walks on $\mathrm{Homeo}_+(S^1)$ that we need to prove Theorem \ref{thm:entropy}. 

The following theorem has already appeared in the literature in several guises, see \cite[Proposition 4.15]{Aoun2011} for probability measures satisfying an exponential moment condition on linear groups acting on projective spaces and \cite[Theorem 1.3]{GelfertSalcedo2023} or \cite[Proposition 4.18]{GorodetskiKleptsyn2021} for finitely supported measures on $\mathrm{Diff}^1_+(S^1)$. The most general version follows from the recent work of I. Choi \cite{Choi2025}, and does not require any assumption on the smoothness of the elements of $G$ nor on the tail decay of the probability measure $\mu$.

\begin{thm} [\cite{Choi2025}] \label{thm:meanXi}
	
	Let $\mu$ be a nondegenerate probability measure on a countable subgroup of $\mathrm{Homeo}_+(S^1)$ acting minimally and proximally on $S^1$. Then there exists $\lambda > 0$ and $N \in \N$ such that for all $n\ge N$ we have
	\begin{equation*} \label{eq:meanXi}
		\sup_{x, y \in S^1} \E\left[d(w_n^{-1}(x), w_n^{-1}(y))\right] \leq e^{-\lambda n}.
	\end{equation*}
\end{thm}

\begin{proof}
	It follows from \cite[Theorem C]{Choi2025} that there exists $\lambda > 1$ such that for all $x,y \in S^1$ and $n \in \N_+$ we have
	\[
	\P\left[ d(w_n^{-1}(x), w_n^{-1}(y)) \leq e^{-\lambda n} \right] \geq 1 - e^{-\lambda n}/\lambda.
	\]
	From this, we obtain
	\[
	\E \left[d(w_n^{-1}(x), w_n^{-1}(y))\right] \leq (1 + 1/\lambda) e^{-\lambda n},
	\]
	which implies the desired inequality.
\end{proof}

The proof of the following corollary uses Theorem \ref{thm:meanXi} and follows steps similar to the proof of \cite[Theorem 4.16]{Aoun2011}.

\begin{corollary}[Exponential convergence in mean to the boundary point] \label{cor:meanBoundary}
	Let $\mu$ be a nondegenerate probability measure on a countable subgroup of $\mathrm{Homeo}_+(S^1)$ acting minimally and proximally on $S^1$. Denote by $\xi:G^{\N}\to S^1$ the boundary map. Then there exist $\lambda>0$ and $N\in \N$ such that for all $n\ge N$ we have
	\[
	\sup_{x \in S^1} \E[d(w_n (x), \xi(\mathbf{w}))] \leq e^{-\lambda n }.
	\] 
\end{corollary}

\begin{proof}
	Let $n,k \in \N_+$ with $n<k$ and let $x,y \in S^1$. We have that
	\[
	\E\left[d\left(w_n(x), \xi(\mathbf{w})\right)\right] \leq \E\left[d\left(w_n(x), w_k(y)\right)\right] + \E\left[d\left((w_k(y), \xi(\mathbf{w})\right)\right].
	\]
	Define a probability measure $\overline{\mu}$ on $G$ as $\overline{\mu}(g) = \mu(g^{-1})$ for $g \in G$. Theorem \ref{thm:meanXi} applied to the random walk driven by $\overline{\mu}$ gives a $\lambda > 0$ such that
	\[
	\sup_{u,v \in S^1} \E\left[ d\left( g_n  \cdots  g_0(u), g_n  \cdots  g_0(v)\right) \right] \leq e^{-\lambda n}
	\] for all large enough $n \in \N_+$. In particular, we deduce that
	\begin{align*}
		\E\left[d\left(w_n(x), w_k(y)\right)\right] & = \sum_{\gamma \in G} \E\left[d\left(w_n(x), w_n  \gamma (y)\right)\right] \mu^{\ast(k-n)}(\gamma) \\
		& \leq \sup_{u,v \in S^1}\E\left[d\left(w_n(u), w_n(v)\right)\right] \\
		& = \sup_{u,v \in S^1} \E\left[ d\left( g_n  \cdots  g_0(u), g_n  \cdots  g_0(v)\right) \right] \le e^{-\lambda n},
	\end{align*} 
	
	and hence conclude that
	\begin{equation*} \label{eq:convExp}
		\sup_{x \in S^1} \E\left[d\left(w_n(x), \xi(\mathbf{w})\right)\right] \leq e^{-\lambda n} + \E\left[d\left(w_k(y), \xi(\mathbf{w})\right)\right].
	\end{equation*} By integrating the above inequality with respect to $\nu$ we conclude that
	\begin{align*}
		\sup_{x \in S^1} \E\left[d\left(w_n(x), \xi(\mathbf{w})\right)\right] &\leq e^{-\lambda n} + \E\left[\int_{S^1}d\left(w_k(y), \xi(\mathbf{w})\right) \dd \nu(y)\right] \\
		& = e^{-\lambda n} + \E\left[\int_{S^1}d(y, \xi(\mathbf{w})) \dd w_k \nu(y)\right].
	\end{align*} But the dominated convergence theorem and Theorem \ref{teo:dkn}, \eqref{it:dkn2} imply that
	\[
	\E\left[\int_{S^1}d(y, \xi(\mathbf{w})) \dd w_k \nu(y)\right] \xrightarrow[k \to \infty]{} \E\left[ \int_{S^1}d(y,\xi(\mathbf{w})) \dd \delta_{\xi(\mathbf{w})}(y)\right] = 0,
	\] so the desired conclusion holds.
\end{proof}

For a probability measure $\mu$ on $G$, we denote by $\overline{\mu}$ the \emph{reflected} probability measure on $G$, defined by $\overline{\mu}(g) = \mu(g^{-1})$ for each $g \in G$.

\begin{proposition} \label{prop:equidistributionXi}
	Let $\mu$ be a nondegenerate probability measure on a countable subgroup of $\mathrm{Homeo}_+(S^1)$ acting minimally and proximally on $S^1$. Denote by $\xi:G^{\N}\to S^1$ the boundary map. Consider the reflected probability measure $\overline{\mu}$ on $G$, and let us denote by $\overline{\nu}$ the unique $\overline{\mu}$-stationary probability measure on $S^1$. Then for any nonempty interval $J \subseteq S^1$ there exists $N \in \N_+$ such that for all $n \geq N$, 
	\[
	\E \left[ \abs{ \left\{ 1 \leq k \leq n :  \xi(\mathbf{w}) \in w_k(J) \right\} }\right] \leq 2 \overline{\nu}(J) n.
	\]
\end{proposition}

\begin{proof}
	Notice that $\overline{\nu}(J) > 0$ because $J$ is nonempty and $\overline{\nu}$ has full support. Since 
	\[
	\E\left[\, \abs{1 \leq k \leq n :  \xi(\mathbf{w}) \in w_k(J) }\,\right] = \sum_{k = 0}^n \P\left[ \xi(\mathbf{w}) \in w_k(J)\right],
	\] it suffices to prove that $\P\left[ \xi \in w_n(J))\right]< 3 \overline{\nu}(J)/2$ for all large enough $n \in \N_+$.  Recall that we write $\sigma \colon G^\N \to G^\N$ for the shift map on the space of sample paths. Since the boundary map $\xi \colon (G^\N, \P) \to (S^1, \nu)$ is $\sigma$-invariant, we have
	\begin{align} \label{eq:overlineNu1} 
		\nonumber \P\left[ \xi(\mathbf{w}) \in w_n(J)\right] =  \P\left[\xi(\sigma^n\mathbf{w}) \in w_n(J)\right] &= \sum_{g \in G} \P\left[\xi(\sigma^n\mathbf{w}) \in w_n(J) \mid w_n = g\right]\mu^{\ast n}(g) \\
		\nonumber & = \sum_{g \in G} \P\left[ \xi(\sigma^n\mathbf{w}) \in g(J) \mid w_n = g\right]\mu^{\ast n}(g) \\
		\nonumber & = \sum_{g \in G} \P\left[ \xi(\mathbf{w}) \in g(J) \right]\mu^{\ast n}(g) \\
		& = \sum_{g \in G} \nu(g(J)) \mu^{\ast n}(g).
	\end{align}
	If $\mu$ were to be symmetric, then $\nu$ would also be $\overline{\mu}$-stationary and we would conclude that $ \P\left[ \xi \in w_n(J)\right] = \nu(J)$, which implies the desired inequality. In the general case, when $\mu$ may not be symmetric, we proceed as follows. For every $\mathbf{w} =(w_n)_{n \geq 0} \in G^\N$ denote by $\overline{\xi}(\mathbf{w}) \in S^1$ the boundary point for the random walk $\{g_1^{-1}\cdots g_n^{-1}\}_{n \geq 0}$, so that $\P$-almost surely
	\[
	(g_1^{-1} \cdots g_n^{-1}) \overline{\nu} \xrightarrow[n \to \infty]{}\delta_{\overline{\xi}(\mathbf{w})}
	\]in the weak-$\ast$ topology. 
	
	Since $\overline{\nu}$ has support equal to $S^1$, whenever $\overline{\xi}(\mathbf{w}) \not \in J$ we have $\mathrm{diam}(g_n \cdots g_1(J)) \xrightarrow[n \to \infty]{} 0$. From this together with the fact that $\nu$ is nonatomic, we obtain
	\begin{equation} \label{eq:nuZero}
		\nu(g_n \cdots g_1(J)) \xrightarrow[n \to \infty]{} 0.
	\end{equation} Next, we have that
	\begin{align} \label{eq:overlineNu2} 
		\nonumber \sum_{g \in G} \nu(g(J)) \mu^{\ast n}(g) & = \int_{G^\N} \nu(g_n \cdots g_1(J)) \dd \P(\mathbf{w}) \\
		\nonumber & \leq  \P\left[ \overline{\xi}(\mathbf{w}) \in J \right] + \int_{\overline{\xi} \not \in J} \nu(g_n \cdots g_1(J)) \dd \P(\mathbf{w}) \\
		& = \overline{\nu}(J) + \int_{\overline{\xi} \not \in J} \nu(g_n \cdots g_1(J)) \dd \P(\mathbf{w}).
	\end{align} The convergence of Equation \eqref{eq:nuZero} together with the dominated convergence theorem show that \[
	\int_{\overline{\xi} \not \in J} \nu(g_n \cdots g_1(J)) \dd \P(\mathbf{w})\xrightarrow[n \to \infty]{} 0,
	\] so the right side of Equation \eqref{eq:overlineNu2} is at most $3\overline{\nu}(J)$/2 for large enough $n \in \N$. Together with Equation \eqref{eq:overlineNu1}, this proves the desired statement.
\end{proof}
\section{There is a linear number of dominating intervals along the walk} \label{sec:linear}
As before, let $\mu$ be a probability measure on a countable subgroup of $\mathrm{Homeo}_+(S^1)$ acting minimally and proximally on $S^1$. Recall that we denote by $(w_n)_{n\geq 0}$ a sample path of the $\mu$-random walk on $G$. 
\begin{definition}
	For every $n, s \in \N_+$ with $1 < s < n$ and each proper interval $J \subseteq S^1$, let us define the random variable $Z^J_{n,s} \in \N$ as the number of times $1 \leq k \leq \lceil n/s \rceil$ such that the interval $w_{ks}(J)$ dominates $w_{js}(J)$ for all $0 \leq j \leq k-1$. That is, we define $Z^J_{n,s} = \sum_{k = 1}^{\lceil n/s \rceil} 1_{B_k}$, where
	\[
	B_k = \{w_{ks}(J) \text{ dominates } w_{js}(J) \text{ for all }0 \leq j \leq k-1\}
	\] for every $1 \leq k \leq \lceil n/s \rceil$. We call the parameter $s \in \N_+$ the \emph{sparsity}.
	
\end{definition}

For every $l \in \N_+$ we denote by $\P_{\mu^{\ast l}}$ the probability measure on $G^\N$ given by the distribution of the trajectories of the $\mu^{\ast l}$-random walk on $G$. Denote by $\E_{\mu^{\ast l}}$ the associated expectation. 
The following proposition guarantees that in expectation there is a linear number of dominated intervals along the trajectory of the random walk.
\begin{proposition} \label{prop:linear}
	Let $\mu$ be a probability measure on a countable subgroup of $\mathrm{Homeo}_+(S^1)$ acting minimally and proximally on $S^1$, and denote by $\nu$ the unique $\mu$-stationary probability measure on $S^1$. Let $I\subseteq S^1$ be a closed interval such that $\nu(I)<1/2$ and let $J\subseteq I$. Then there exist $s,N \in \N_+$ and $0 < c < 1$ such that for all $l \in \N_+$ and every $n \ge N$ we have $\E_{\mu^{\ast l}}[Z^J_{n,s}]\geq cn$.
\end{proposition}
The proof of this proposition will follow from the next two lemmas.
\begin{lemma} \label{lemma:greedy}
	Consider the same hypotheses as in Proposition \ref{prop:linear}. Then for any $s,l \in \N_+$ there is $N\ge 1$ such that for all $n\ge N$ we have
	\[
	\E_{\mu^{\ast l}}[Z^J_{n,s}] \geq \frac{n}{2s}\P_{\mu^{\ast l}}\left[ w_{js}(J) \emph{ dominates } J \emph{ for all } j \geq 1 \right].
	\]
\end{lemma}

\begin{proof}
	For $l \in \N_+$ we have
	\begin{align*}
		\E_{\mu^{\ast l}}[Z^J_{n,s}] & = \sum_{k = 1}^{\lceil n/s \rceil} \P_{\mu^{\ast l}}\left[w_{ks}(J) \text{ dominates } w_{js}(J) \text{ for all }0 \leq j \leq k-1 \right] \\
		& = \sum_{k = 1}^{\lceil n/s \rceil} \P_{\mu^{\ast l}}\left[ g_{js + 1} g_{js + 2} \cdots g_{ks}(J) \text{ dominates } J \text{ for all } 0 \leq j \leq k-1 \right] \\
		& = \sum_{k = 1}^{\lceil n/s \rceil} \P_{\mu^{\ast l}}\left[ g_{1} g_{2} \cdots g_{(k-j)s}(J) \text{ dominates } J \text{ for all } 0 \leq j \leq k-1 \right] \\
		& = \sum_{k = 1}^{\lceil n/s \rceil} \P_{\mu^{\ast l}}\left[ w_{js}(J) \text{ dominates } J \text{ for all } 1 \leq j \leq k \right],
	\end{align*} where the second to last equality follows from the fact that the increments $(g_j)_{j\ge 1}$ are independent and identically distributed.
	
	For every $k=1,2,\ldots,\lceil n/s \rceil$ let us denote by $D_k$ the event where $w_{js}(J)$ dominates $J$ for all $1\leq j \leq k$. Note that $D_{k+1}\subseteq D_k$ for each $k=1,2,\ldots,\lceil n/s \rceil-1$, and therefore we have that 
	\[
	\P_{\mu^{\ast l}}\left[D_k \right]\xrightarrow[k\to \infty]{}\P_{\mu^{\ast l}}\left[ w_{js}(J) \text{ dominates } J \text{ for all }j\geq 1 \right].
	\] From this, we also obtain 
	\[
	\left \lceil \frac{n}{s} \right \rceil^{-1} \sum_{k=1}^{\lceil n/s \rceil} \P_{\mu^{\ast l}}\left[D_k\right]\xrightarrow[n\to \infty]{}\P_{\mu^{\ast l}}\left[ w_{js}(J) \text{ dominates } J \text{ for all }j\geq 1 \right],
	\]
	which implies the desired inequality.
\end{proof}

\begin{lemma}\label{lemma: lower bound prob. dominating} Consider the same hypotheses as in Proposition \ref{prop:linear}. Then there exists a sparsity $s \in \N_+$ such that for all $l \in \N_+$ and every interval $J \subseteq I$ we have
	\[
	\P_{\mu^{\ast l}}\left[ w_{js}(J) \emph{ dominates } J \emph{ for all }j \geq 1\right] \geq 1/24.
	\]
\end{lemma}

\begin{proof} Note that the probability measure $\nu$ on $S^1$, which is the unique $\mu$-stationary probability measure on $S^1$, is also the unique $\mu^{\ast l}$-stationary measure on $S^1$ for each $l\ge 1$. Since $\nu$ is nonatomic we have that
	\[
	\nu(\{\xi \in S^1 :  0 < d(\xi,I)<\epsilon\}) \xrightarrow[\epsilon \to 0]{} 0.
	\] We see from this that there exists $\epsilon>0$, that does not depend on $l$, such that
	\[
	\nu(\{\xi \in S^1 :  d(\xi,I)>\epsilon\})\ge  (1-\nu(I))/2.
	\] 
	Let us define
	\( \Xi = \{\xi \in S^1 :  d(\xi,I)>\epsilon\}\), and recall that we are supposing $\nu(I)<1/2$. Together with the above, this implies that $\nu(\Xi)> 1/4$.
	
	Let $\lambda > 0$ be such that the conclusion of Corollary \ref{cor:meanBoundary} is verified for $\P = \P_{\mu^{\ast l}}$, and choose a sparsity $s$ such that $e^{-\lambda s/2} < \min\{\epsilon/8, 1/5\}$. Since
	\begin{equation} \label{eq:uniformity}
		\sup_{x \in S^1} \E_{\mu^{\ast l}} \left[ d\left( w_n(x), \xi(\mathbf{w}) \right) \right] = \sup_{x \in S^1} \E \left[ d(w_{ln}(x), \xi(\mathbf{w})) \right] \leq e^{-\lambda ln} \leq e^{-\lambda n}
	\end{equation} for all $n \in \N_+$, we can choose $\lambda$ and $s$ uniform in $l$.
	
	Denote by $l$ (resp.\ $r$) the left (resp.\ right) endpoint of the interval $J$, so that we have $J = [l,r]$. We claim that if the interval $w_{js}(J)$ does not dominate $J$ for some $j \geq 1$, then for each $\xi \in \Xi$ we have $\max\{d(w_{js}(l), \xi), d(w_{js}(r), \xi)\} \geq \epsilon$. Indeed, if $w_{js}(J)$ does not dominate $J$, then either $w_{js}(J) \subseteq J$ or $\{w_{js}(l), w_{js}(r)\}\cap J\neq \varnothing$. In both cases we obtain that $\{w_{js}(l), w_{js}(r)\}\cap I\neq \varnothing$, and hence there is an endpoint of the interval $w_{js}(J)$ at distance at least $\epsilon$ from $\Xi$ (see Figure \ref{fig:epsilonDistance}).
	
	\begin{figure}[h!]
		\centering
		\begin{tikzpicture}[scale=1]
			\draw (0,0) circle (2.2);
			
			\draw[thick, yellow] ([shift=(-22.5:2)]0,0) arc (-22.5:202.5:2) node[black, below, pos=.5] {$\Xi$};
			\draw[thick, red] ([shift=(210:2)]0,0) arc (210:330:2) node[black, above, pos=.5] {$I$};
			\draw[thick, red] ([shift=(250:2.4)]0,0) arc (250:290:2.4) node[black, below left, pos=.5] {$J$};
			\draw[thick, red] ([shift=(280:2.6)]0,0) arc (280:390:2.6) node[black, below right, pos=.5] {$w_{js}(J)$};
		\end{tikzpicture}
		\caption{The interval $w_{js}(J)$ does not dominate $J$.}
		\label{fig:epsilonDistance}
	\end{figure}
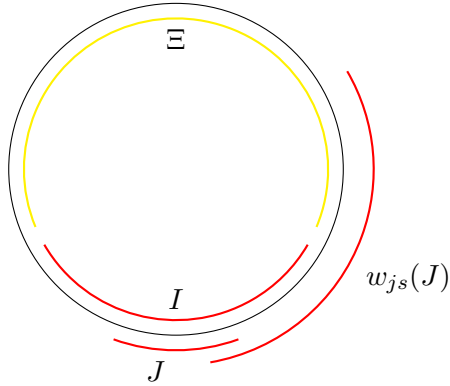

	For each $j \geq 1$ let us denote by $A_j$ the event where $w_{js}(J)$ does not dominate $J$. By the above paragraph, for every $j \geq 1$ we have
	\begin{align*} 
		\nu(\Xi) \P_{\mu^{\ast l}} \left[ A_j \mid \xi \in \Xi\right] & \leq \nu(\Xi) \P_{\mu^{\ast l}} \left[  \max\{d(w_{js}(l), \xi), d(w_{js}(r), \xi)\} > \epsilon \mid \xi \in \Xi \right] \\
		& \leq \P_{\mu^{\ast l}}\left[ \max\{d(w_{js}(l), \xi), d(w_{js}(r), \xi)\} > \epsilon\right] \leq e^{-\lambda js}\frac{2}{\epsilon},
	\end{align*}
	where in the last inequality we used Equation \eqref{eq:uniformity} together with the Markov inequality.
	
	Since $e^{-\lambda s/2} < \epsilon \nu(\Xi)/2$, we see that
	\[
	\P_{\mu^{\ast l}} \left[ A_j \mid \xi \in \Xi \right] \leq e^{-\lambda js}\frac{2}{\epsilon \nu(\Xi)} \leq e^{-\lambda s/2} e^{-\lambda (j-1)s}.
	\]

	Hence
	\[
	\P_{\mu^{\ast l}} \left[\bigcup_{j \geq 1} A_j  \mid \xi \in \Xi \right] \leq e^{-\lambda s/2} \sum_{j \geq 1} e^{-\lambda(j-1)s}  =  \frac{e^{-\lambda s/2}}{1 - e^{-\lambda s}}
	\] which is at most $5/24$ since $e^{-\lambda s/2} \leq 1/5$ and $e^{-\lambda s}\leq 1/25$. Finally, we obtain
	\begin{align*}
		\P_{\mu^{\ast l}}\left[ \bigcup_{j \geq 1} A_j \right] & = \P_{\mu^{\ast l}}\left[ \bigcup_{j \geq 1} A_j \mid \xi(\mathbf{w}) \in \Xi \right]\P_{\mu^{\ast l}}[\xi(\mathbf{w}) \in \Xi]\ + \\ &\hspace{20pt} + \P_{\mu^{\ast l}}\left[ \bigcup_{j \geq 1} A_j \mid \xi(\mathbf{w}) \not \in \Xi \right]\P_{\mu^{\ast l}}[\xi(\mathbf{w}) \not \in \Xi] \\
		& \leq \frac{5}{24} + \nu(S^1\setminus \Xi) \leq \frac{5}{24} + \frac{3}{4} = \frac{23}{24},
	\end{align*} which proves the lemma.
\end{proof}

\begin{proof}[Proof of Proposition \ref{prop:linear}]
	Lemma \ref{lemma:greedy} together with Lemma \ref{lemma: lower bound prob. dominating} imply the statement of Proposition \ref{prop:linear} by setting $c = \frac{1}{48 s}$.
\end{proof}

\section{Proof of Theorem \ref{thm:entropy}} \label{sec:proofA}
We consider, as in the statement of Theorem \ref{thm:entropy}, a countable group $G$ acting proximally, minimally, and topologically nonfreely on $S^1$ and $\mu$ a nondegenerate probability measure on $G$. By Lemma \ref{lemma: lazy boundaries}, we may assume that $\mu(e_G) > 0$.

We denote by $\nu$ the unique $\mu$-stationary measure on $S^1$, and we denote by $\overline{\nu}$ the unique $\overline{\mu}$ stationary measure where $\overline{\mu}$ is the nondegenerate probability measure on $G$ given by $\overline{\mu}(g) = \mu(g^{-1})$ for each $g \in G$.

Fix an interval $I \subset S^1$ such that $\nu(I) < 1/2$. Using Proposition \ref{prop:linear} we can find $0< c < 1$ and $s \in \N_+$ such that $\E_{\mu^{\ast l}}[Z^J_{n,s}] \geq c n$ for all $l \in \N_+$, every interval $J \subseteq I$ and all sufficiently large $n \in \N_+$. Since
\[
\E_{\mu^{\ast sl}}[Z^J_{n,1}] = \E_{\mu^{\ast l}}[Z^J_{ns,s}] \geq c ns \geq cn,
\] for all $l \in \N_+$ and sufficiently large $n \in \N_+$, we can replace once and for all $\mu$ by $\mu^{\ast sl}$ with some large $l \in \N_+$ (and drop the subscript $\mu^{*l}$ from $\P, \E$) so that 
\begin{itemize}
	\item $\mathrm{supp}(\mu)$ contains an element $a \in G \setminus \{e_G\}$ with $\mathrm{supp}(a) \subseteq I$, such that the smallest subinterval $J \subseteq I$ that contains $\mathrm{supp}(a)$ satisfies $\overline{\nu}(J) < c/8$, and
	\item $\E[Z^J_{n,1}] \geq cn $ for all large $n \in \N_+$.
\end{itemize} Note that the stationary measures $\nu$, $
\overline{\nu}$ on $S^1$ do not change after doing this replacement.

\begin{lemma}\label{lemma:Wn}
	For every $n \in \N_+$, denote by $W_n \in \N_+$ the random variable that counts the number of times $1 \leq k \leq n$ such that
	\begin{itemize}
		\item the interval $w_{k}(J)$ dominates $w_{j}(J)$ for all $0 \leq j \leq k-1$,
		\item $w_{k}^{-1}(\xi(\mathbf{w}))$ does not belong to $J$, and
		\item the increment $g_{k+1}$ is in $\{a, e_G\}$. 
	\end{itemize}
	Then there exists $0 < c' < 1$ such that $\E[W_n] \geq c' n$ for sufficiently large $n \in \N_+$.
\end{lemma} 
\begin{proof}
	For every $n \in \N_+$, denote by $\widetilde{W}_n \in \N$ the random variable that counts the number of times $1 \leq k \leq n$ such that
	\begin{itemize}
		\item the interval $w_{k}(J)$ dominates $w_{j}(J)$ for all $0 \leq j \leq k-1$, and
		\item $w_{k}^{-1}(\xi(\mathbf{w}))$ does not belong to $J$.
	\end{itemize}
	By Corollary \ref{prop:equidistributionXi}, we see that 
	\begin{equation} \label{eq:linear2}
		\frac{\E\left[\abs{\left\{ 1 \leq k \leq n : w_k^{-1}(\xi(\mathbf{w})) \not \in J \right\} }\right]}{n} \geq 1-2\overline{\nu}(J) \geq 1 -\frac{c}{4},
	\end{equation} for all sufficiently large $n \in \N_+$. The bound $\E[Z^J_{n,1}] \geq c n$ and Equation \eqref{eq:linear2} imply that there exists $c'' \in (0,1)$ such that $\E[\widetilde{W}_n] \geq c'' n$ for all sufficiently large $n \in \N_+$.
	
	Notice that
	\[
	\widetilde{W}_n = \sum_{k = 1}^n 1_{\widetilde{C}_k} \quad \text{and} \quad W_n = \sum_{k = 1}^n 1_{C_k}, 
	\] where
	\[
	\widetilde{C}_k = \{w_k(J) \text{ dominates } w_j(J) \text{ for }0 \leq j \leq k-1, \text{ and }w_k^{-1}(\xi(\mathbf{w})) \not \in J\}
	\] and 
	\[
	C_k = \widetilde{C}_k \cap \{g_{k+1} \in \{a, e_G\}\}
	\] for every $1 \leq k \leq n$. Since the event $\{g_{k+1} \in \{a, e_G\}\}$ is independent from $\widetilde{C}_k$ under $\P$, we deduce that
	\[
	\E[W_n] = (\mu(a) + \mu(e_G)) \E[ \widetilde{W}_n ] \geq (\mu(a) + \mu(e_G)) c'' n
	\]
	for all sufficiently large $n \in \N_+$. The statement from the lemma then holds for the value $c' = (\mu(a) + \mu(e_G)) c''$.
\end{proof}

We recall the following basic fact about random variables, that we will use below.
\begin{lemma}\label{lemma:pigeon}
	Let $0 \leq X \leq 1$ be a real-valued random variable with mean $\E[X] > \lambda > 0$.
	Then
	\[
	\P[ X > \lambda/2] \geq \lambda/2.
	\]
\end{lemma}
\begin{proof}
	The statement follows from the inequality
	\[
	\lambda < \E[X] \leq \lambda/2 \P[X \leq \lambda/2] + \P[ X > \lambda/2] \leq \lambda/2 + \P[ X > \lambda/2]. \qedhere
	\]
\end{proof}
Finally, we present the proof of Theorem \ref{thm:entropy}.
\begin{proof}[Proof of Theorem \ref{thm:entropy}]
	Just as in the statement of Lemma \ref{lemma:Wn}, for every $n\ge 1$ let us denote by $W_n \in \N_+$ the random variable that counts the number of times $1 \leq k \leq n$ such that
	\begin{itemize}
		\item the interval $w_{k}(J)$ dominates $w_{j}(J)$ for all $0 \leq j \leq k-1$,
		\item $w_{k}^{-1}(\xi(\mathbf{w}))$ does not belong to $J$, and
		\item the increment $g_{k+1}$ is in $\{a, e_G\}$. 
	\end{itemize}
	
	For $\nu$-almost every $\xi \in S^1$ and $n \in \N_+$ we apply Lemma \ref{lemma:pigeon} to the random variable \[\mathbf{w} \in (G^\N, \P^\xi) \mapsto \frac{W_n}{n} \in [0,1],\] and deduce that
	\[
	\P^\xi\left[ \frac{W_n}{n} > \frac{\E^\xi[W_n]}{2n}\right] \geq \frac{\E^\xi [W_n]}{2n}.
	\] Now consider the random variable
	\[
	\xi \in (S^1, \nu) \mapsto \frac{\E^\xi[W_n]}{n} \in [0,1]
	\] and apply Lemmas \ref{lemma:Wn} and \ref{lemma:pigeon} to see that $\nu(\Xi_n)>c'/2$, where
	\[
	\Xi_n \coloneqq \left \{\xi \in S^1  \, : \,  \frac{\E^\xi[W_n]}{n} \geq \frac{c'}{2} \right \}.
	\]
	From this, we conclude that  
	\begin{equation} \label{eq:Wn}
		\P^\xi \left[ \frac{W_n}{n} > \frac{c'}{2} \right] \geq \frac{c'}{2}
	\end{equation} for $\nu$-almost every $\xi \in \Xi_n$.
	
	For every $n \in \N_+$ and $\nu$-almost every $\xi \in \Xi_n$, consider the set of infinite trajectories $\mathbf{w} = (w_n)_{n \geq 0}\in G^\N$ such that $W_n(\mathbf{w})/n > c'/2$. To each such sample path $\mathbf{w}$ we associate a maximal set of indices $1 \leq i_1 < \cdots < i_k \leq n$ of size $k = k(\mathbf{w})$ such that for every $1 \leq r \leq k$ we have that
	\begin{itemize}
		\item the interval $w_{i_r - 1}(J)$ dominates $w_l(J)$ for all $0 \leq l < i_r - 1$, and
		\item $g_{i_r} \in \{a, e_G\}$ and $w_{i_r}^{-1}(\xi) \not \in J$.
	\end{itemize}
	Define $\mathbf{x}$ as the $(n-k)$-tuple consisting of all increments of $\mathbf{w}$ at times instants in $\{1,\ldots, n\} \setminus \{i_1, \ldots, i_k\}$. By definition, the collection $q(\mathbf{w}) = (\mathbf{x}, i_1,\ldots, i_k)$ is $\xi$-good, and we have $k(\mathbf{w}) \geq W_n > c'n/2$. 
	
	Denote by $A_{n,\xi}$ the set of collections obtained in this way. We claim that the collections $A_{n,\xi}$ satisfy the conditions of Theorem \ref{thm:relEntropy}. Indeed, since $A_{n,\xi}$ is composed of $\xi$-good collections, thanks to Proposition \ref{proposition:rconditions} we have that Conditions \eqref{item:rentropy3}, \eqref{item:rentropy4} and \eqref{item:rentropy2} are satisfied. Moreover, for every $q = (\mathbf{x}, i_1, \ldots, i_k) \in A_{n,\xi}$ we have $k \geq c'n/2$ by the previous paragraph. Therefore, Condition \eqref{item:rentropy1} is also satisfied. Finally, from Inequality \eqref{eq:Wn} we get that
	\[
	\P^\xi \left[ \bigcup_{q \in A_{n,\xi}} \bigcup_{\mathbf{y} \in T^{a,e_G}(q)} [\mathbf{y}] \right] \geq c'/2,
	\] so Condition \eqref{item:rentropy5} also holds. The hypotheses of Theorem \ref{thm:relEntropy} are satisfied, and hence we have finished the proof of Theorem \ref{thm:entropy}.
\end{proof}
\section{Proof of Theorem \ref{thm:lamps}} \label{sec:harmonic}
Recall that the group $\mathrm{PAff}_+(S^1)$ of \emph{piecewise affine} orientation-preserving homeomorphisms of $S^1 \cong \R/\Z$ is the group of all $g \in \mathrm{Homeo}_+(S^1)$ such that there exists a finite subset $D \subset S^1$ such that $g$ restricted to every connected component $C$ of $S^1 \setminus D$ is of the form $g(x) = a x + b$ for some $a > 0$ and $b \in \R/\Z$. Thus, for every $g \in \mathrm{PAff}_+(S^1)$ the derivative $g'$ is defined outside a finite set and is locally constant. The points of $S^1$ where the derivative of $g$ is not defined are called the \emph{breakpoints} of $g$.

In this section $G$ is a countable subgroup of $\mathrm{PAff}_+(S^1)$ acting minimally, proximally and topologically nonfreely on $S^1$, and $\mu$ is a nondegenerate probability measure on $G$ such that $\sum_{g \in G} \mu(g) \abs{\cC_g}$ is finite, where $\abs{\cC_g}$ is the number of breakpoints of $g \in G$. We denote by $\mathbf{Br}\subset S^1$ the countable set of breakpoints of elements in $G$.

\subsection{The breakpoint boundary}
In this subsection we adapt V. Kaimanovich's construction of a $\mu$-boundary of Thompson's group $F$ \cite{Kaimanovich2017} to the context above. The arguments are analogous to those that appear in \cite{Kaimanovich2017} and their extension by B. Stankov \cite{Stankov2021} for probability measures with a finite first moment.

Given an element $g \in G$, define a finitely supported function $\cC_g \colon \mathbf{Br} \to \R$ by setting 
\[
\cC_g(x) = \log\left((g^{-1})'(x^+)\right) - \log\left((g^{-1})'(x^-)\right)
\] for $x \in \mathbf{Br}$, where $(g^{-1})'(x^+)$ (resp. $(g^{-1})'(x^-)$ is the left (resp. right) derivative of $g^{-1}$ at $x$. That is, $\cC_g(x)$ is the derivative jump of $g^{-1}$ at $x$. This definition differs slightly from those used in \cite{Kaimanovich2017, Stankov2021}, but this difference is necessary since we consider \emph{right} random walks and the \emph{left} action of $G$ on $S^1$.

Denote the set of all (not necessarily finitely supported) functions $\mathbf{Br} \to \R$ by $\R^{\mathbf{Br}}$. Define a left action of $G$ on $\R^{\mathbf{Br}}$ by
\[
(g, \cC) \mapsto \left(S_g \cC \colon x \in \mathbf{Br} \mapsto \cC(g^{-1}(x))\right).
\] By the chain rule, we have
\[
\cC_{gh}(x) = \log\left((h^{-1})'(g^{-1}(x)^+)\right) - \log\left((h^{-1})'(g^{-1}(x)^-)\right) + \log\left((g^{-1})'(x^+)\right) - \log\left((g^{-1})'(x^-)\right)
\]for all $g,h \in G$ and $x \in \mathbf{Br}$, so that 
\begin{equation} \label{eq:equivariance}
	\cC_{gh} = \cC_g + S_g\cC_h.
\end{equation}
Let us define a second left action of $G$ on $\R^{\mathbf{Br}}$ by $(g, \cC) \mapsto \cC_g + S_g\cC$, so we have $\cC_{gh} = g.\cC_h$ for all $g,h \in G$. This is the action on $\R^{\mathbf{Br}}$ that will define a nontrivial boundary for $G$.

To prove the transience of the random walks on $G$-orbits of elements of $\mathbf{Br}$ we emulate \cite[Theorem 25]{Kaimanovich2017}, for which we need a comparison lemma for Markov operators due to \cite{BaldiLohouePeyriere1977}; see also the proposition at the end of Section 4 in \cite{Varopoulos1983} for a more general version of this result.
\begin{proposition} \label{prop:comparison}
	Let $P_1(\cdot, \cdot), P_2(\cdot, \cdot)$ be doubly stochastic kernels on a countable set $X$ such that $P_2(\cdot, \cdot)$ is symmetric and there exists $\epsilon > 0$ such that
	\[
	P_1(x,y) \geq \epsilon P_2(x,y) \text{ for all }x,y\in X.
	\]
	Then the Markov process determined by $P_1$ and started at $x \in X$ is transient if the Markov process determined by $P_2$ and started at $x \in X$ is transient.
\end{proposition}

\begin{lemma}\label{lemma:transience}
	For every $x \in S^1$ the $\mu$-random walk on $\mathrm{Orb}_G(x)$ started at $x$ is transient.
\end{lemma} 
\begin{proof}
	Fix $x \in S^1$. Since $G$ acts proximally and minimally, using Proposition \ref{proposition:small support} we can find $f,g \in G$ such that there are disjoint intervals $I_1, I_2, J_1, J_2 \subset S^1$ with $x \not \in I_1 \cup I_2 \cup J_1 \cup J_2$ and 
	\[
	f(S^1 \setminus I_2) \subseteq I_1,\quad g(S^1 \setminus J_2) \subseteq J_1.
	\]
	
	By Klein's ping-pong lemma, $f$ and $g$ generate a free subgroup of $G$ and $\langle f, g\rangle$ acts freely on $\mathrm{Orb}_{\langle f, g \rangle}(x)$. Let $\tilde{\mu}$ be the uniform measure on $\{f,f^{-1}, g, g^{-1}\}$. The $\tilde{\mu}$-random walk on $\mathrm{Orb}_{\langle f,g\rangle}(x)$ starting at $x$ is transient since it corresponds to a simple random walk on an infinite tree of valence $4$. Let $n \in \N$ be such that $f,g\in \supp{\mu^{\ast n}}$. Then there exists $\epsilon > 0$ such that $\mu^{\ast n} \geq \epsilon \tilde{\mu}$, and we obtain from Proposition \ref{prop:comparison} that the $\mu^{\ast n}$-random walk on $\mathrm{Orb}_G(x)$ started at $x$ is transient. We conclude that the $\mu$-random walk on $\mathrm{Orb}_G(x)$ starting at $x$ is transient too.
\end{proof}

The following lemma is reminiscent of the stabilization of lamp configurations for random walks on wreath products; see the references in Subsection \ref{subsec:background} and also \cite[Lemma 7.2]{Stankov2021}. This is the only point in the construction of the breakpoint boundary where the moment condition on $\mu$ is used. For $\cC \in \R^{\mathbf{Br}}$ we denote $\mathrm{supp}(\cC) = \{x \in \mathbf{Br} \mid \cC(x) \neq 0\}$.

\begin{lemma}
	Suppose that $\sum_{g \in G} \mu(g) \abs{\cC_g}<\infty$. Then $\P$-almost surely and every $x\in \mathbf{Br}$ there is $N\ge 1$ such that $\cC_{w_n}(x)=\cC_{w_{n+1}}(x)$ for each $n\ge N$. Hence, the configurations $(\cC_{w_n})_{n\ge 0}$ converge pointwise to a map $\cC_\infty(\mathbf{w}) \in \R^{\mathbf{Br}}$.
\end{lemma}

\begin{proof} The equation
	\[
	\cC_{w_{n+1}} = \cC_{w_n} + S_{w_n}\cC_{g_{n+1}}
	\] implies that, for every $x \in \mathbf{Br}$, $\cC_{w_{n+1}}(x)= \cC_{w_n}(x)$ if and only if $w_{n}^{-1}(x)\notin \mathrm{supp}(\cC_{g_{n+1}})$. Thus, the configurations $\cC_{w_n}$ stabilize as $n \to \infty$ to some $\cC_{\infty}(\mathbf{w}) \in \R^{\mathbf{Br}}$ if for every $x \in \mathbf{Br}$, we have $w_n^{-1}(x) \in \mathrm{supp}(\cC_{g_{n+1}})$ for only finitely many $n \in \N$. By the Borel-Cantelli lemma, this holds whenever
	
	\begin{equation*}\label{eq:borelCantelli}
		\sum_{n \geq 0} \P\left[ w_n^{-1}(x) \in \cC_{g_{n+1}}\right] = \sum_{n \geq 0} \sum_{g \in G} \mu(g) \P\left[ w_n^{-1}(x) \in \cC_g\right] = \sum_{n \geq 0} \sum_{g \in G} \sum_{y \in \cC_g}\mu(g) \P\left[ w_n^{-1}(x) = y\right]
	\end{equation*} is finite for all $x \in \mathbf{Br}$. 
	
	The $\mu$-random walk $(w_n)_{n\geq 0}$ on $G$ induces a Markov chain on $\mathbf{Br}$ with transition probabilities $p(x,y) = \P\left[ w_1^{-1}(x) = y\right]$. Denote by $\overline{p}$ its reflected kernel, defined by $\overline{p}(x,y) = p(y,x)$ for all $x,y \in \mathbf{Br}$. Notice that for $x,y \in \mathbf{Br}$, the quantity $\sum_{n \geq 0} \overline{p}^{\ast n}(y,x)$ is the expected number of visits to $x$ of the Markov chain defined by $\overline{p}$ and starting at $y$. Hence
	\[
	\sum_{n \geq 0} \overline{p}^{\ast n}(y,x) = q(y,x) \sum_{n \geq 0} \overline{p}^{\ast n}(x,x)
	\] where $q(y,x)$ is the probability that the Markov chain defined by $\overline{p}$ and started at $y$ hits $x$. In particular,
	\[
	\sum_{n \geq 0} \overline{p}^{\ast n}(y,x) \leq \sum_{n \geq 0} \overline{p}^{\ast n}(x,x),
	\] so an upper bound for $\sum_{n \geq 0} \P\left[ w_n^{-1}(x) \in \cC_{g_{n+1}}\right] $ is given by
	\[
	\sum_{g \in G} \sum_{y \in \cC_g} \mu(g) \left( \sum_{n \geq 0} \overline{p}^{\ast n}(x,x) \right) = \left( \sum_{n \geq 0} \overline{p}^{\ast n}(x,x) \right) \left( \sum_{g \in G} \mu(g) \abs{\cC_g}\right).
	\] By Lemma \ref{lemma:transience}, the sum $\sum_{n \geq 0} \overline{p}^{\ast n}(x,x)$ is finite for any $x \in \mathbf{Br}$, and by hypothesis the term $\sum_{g \in G} \mu(g) \abs{\cC_g}$ is also finite. This proves the lemma.
\end{proof}

Denote by $\widetilde{\nu}$ the pushforward measure of $\P$ through $\cC_\infty \colon G^\N \to \R^{\mathbf{Br}}$. Then the space $(\R^{\mathbf{Br}}, \widetilde{\nu})$ is a $\mu$-boundary, that we call the \emph{breakpoint boundary}. One can prove the non-triviality of this boundary by following steps similar to those in the proof of \cite[Lemma 7.3]{Stankov2021}, or alternatively, this will follow from the proof of Theorem \ref{thm:lamps} in the next subsection.

\subsection{Bounded harmonic functions not coming from $(S^1, \nu)$} Recall that we denote by $(S^1, \nu)$ the $\mu$-boundary coming from the natural action of $G$ on the circle. In this subsection, we provide a family of harmonic functions defined through the breakpoint boundary $(\R^{\mathbf{Br}},\tilde{\nu})$ that cannot be obtained from $(S^1,\nu)$ via the Poisson transform. 
\begin{lemma} \label{lemma:gn}
	Let $G$ be a countable subgroup of $\mathrm{PAff}_+(S^1)$ whose action on $S^1$ is minimal, proximal and topologically nonfree, and let $\mu$ be a nondegenerate probability measure on $G$ such that $\sum_{g \in G} \mu(g) \abs{\mathbf{Br}_g}<\infty$. Then for every $n \in \N_+$ there exists a $\mu$-harmonic function $f_n \colon G \to [0,1]$ with $f_n(e_G) > 1/2$ and an element $a_n \in G$ such that
	\[
	\mathrm{diam}(\mathrm{supp}(a_n)) \xrightarrow[n \to \infty]{} 0 \quad \text{ and } \quad f_n(a_n) \xrightarrow[n \to \infty]{} 0.
	\]
\end{lemma}

\begin{proof}
	Consider an element $a \in G \setminus \{e_G\}$ such that $\mathrm{supp}(a)$ is strictly contained in $S^1$ (such an element exists because the action of $G$ is topologically nonfree). Fix $y \in \partial ( \mathrm{supp}(a) )$, so that $y \in \mathbf{Br}$. Since the action of $G$ on $S^1$ is minimal and proximal, there exists a sequence $\{t_n\}_{n \geq 0} \subset G$ such that $\mathrm{diam}(t_n(\mathrm{supp}(a))) \leq 1/n$. Since the measure $\tilde{\nu}$ is nonzero, for every $n \in \N_+$ there exists a bounded open set $U_n \subset \R$ such that the bounded function $f_n \colon G \to [0,1]$ defined on $g \in G$ by
	\[
	f_n(g) = \P_g\left[ \left\{ \mathbf{w} \in G^\N : \cC_\infty(\mathbf{w})(t_n(y)) \in U_n\right\} \right]
	\] satisfies $f_n(e_G) > 1/2$. Note that the event \(\{\mathbf{w} \in G^\N  :\cC_\infty(\mathbf{w})(t_n(y)) \in U_n \}\) is shift-invariant up to $\P$-measure zero, and hence the function $f_n$ is $\mu$-harmonic.
	
	For each $n \in \N_+$ define $b_n = t_na t_n^{-1}$. Then we have 
	\(
	f_n(b_n^j) = \P\left[ \cC_\infty(b_n^j\mathbf{w})(t_n(y)) \in U_n \right]
	\) 
	for every $j \in \N_+$. Using Equation \eqref{eq:equivariance} we get
	\begin{equation} \label{eq:cC}
		\cC_\infty(b_n^j\mathbf{w})(t_n(y)) = \cC_{b_n^j}(t_n(y)) + S_{b_n^j} \cC_\infty(\mathbf{w})(t_n(y)) = \cC_{b_n^j}(t_n(y)) + \cC_\infty(\mathbf{w})(t_n(y)),
	\end{equation} where in the last equality we used that $b_n^j$ fixes $t_n(y)$. By iterating Equation \eqref{eq:equivariance} and using that $a^{-j}$ fixes $y$, we see that
	\begin{align*}
		\cC_{b_n^j}(t_n(y)) = \cC_{t_n a^j t_n^{-1}}(t_n(y)) & = \cC_{t_na^j}(t_n(y)) + S_{t_na^j}\cC_{t_n^{-1}}(t_n(y)) \\ 
		& = \cC_{t_na^j}(t_n(y)) + \cC_{t_n^{-1}}(a^{-j}(y)) \\
		& = \cC_{t_na^j}(t_n(y)) + \cC_{t_n^{-1}}(y) \\
		& = \cC_{t_n}(t_n(y)) + S_{t_n}\cC_{a^j}(t_n(y)) +  \cC_{t_n^{-1}}(y) \\
		& = \cC_{t_n}(t_n(y)) + \cC_{a^j}(y) +  \cC_{t_n^{-1}}(y)  = j\cC_{a}(y).
	\end{align*}
	The above together with Equation \eqref{eq:cC} shows that
	\( f_n(b_n^j) = \P\left[ j\cC_a(y)  + \cC_\infty(\mathbf{w})(t_n(y))  \in U_n \right].\) Next, note that $\cC_a(y)\neq 0$ by the choice of $y$. Since $U_n$ is bounded, there exists $j_n \in \N_+$ sufficiently large so that $f_n(b_n^{j_n}) < 1/n$. Denote $a_n \coloneqq b_n^{j_n}$. Then the diameter of
	\(
	\mathrm{supp}(a_n) = \mathrm{supp}(b_n) = t_n(\mathrm{supp}(a))
	\) goes to $0$ as $n\to \infty$, and hence $f_n(a_n)\xrightarrow[n\to \infty]{}0$.
\end{proof}
We can now present the proof of Theorem \ref{thm:lamps}. 
\begin{proof}[Proof of Theorem \ref{thm:lamps}] 
	Looking for a contradiction, let us suppose that the breakpoint boundary is a $G$-equivariant quotient of $(S^1, \nu)$. For every $n \in \N_+$ consider $\mu$-harmonic functions $f_n \colon G \to [0,1]$ and elements $a_n \in G$ as in Lemma \ref{lemma:gn}. Then there exist functions $h_n \in L^\infty(S^1, \nu)$ such that
	\[
	f_n(g) = \int_{S^1}h_n(g(x)) \dd \nu(x) \text{ for all }g\in G.
	\] Since the Poisson transform is an isometry (Theorem \ref{thm: isometry}), we have that $\norm{h_n}_\infty \leq 1$.
	
	Write $I_n = \mathrm{supp}(a_n)$, so $h_n(a_n(x)) = h_n(x)$ whenever $x \not  \in I_n$ and 
	\begin{align*}
		\abs{f_n(a_n) - f_n(e_G)} & \leq \int_{S^1 \setminus I_n}\abs{h_n(a_n(x)) - h_n(x)} \dd \nu(x) + \int_{I_n} \abs{h_n(a_n(x)) - h_n(x)} \dd \nu(x) \\
		& = \int_{I_n} \abs{h_n(a_n(x)) - h_n(x)} \dd \nu(x) \leq 2 \nu(I_n) \norm{h_n}_\infty \leq 2 \nu(I_n)\xrightarrow[n\to \infty]{}0.
	\end{align*} However, from Lemma \ref{lemma:gn} we have that $\liminf_{n\to \infty}\abs{f_n(a_n) - f_n(e_G)} \geq 1/2$. This is a contradiction.
\end{proof}

\begin{rem}
	When $G$ is Thompson's group $T$ we can provide a single $\mu$-harmonic function that does not arise from $(S^1,\nu)$. Indeed, in this case $\mathbf{Br} = \Z[1/2]/\Z$ and, after defining configurations using logarithms in base 2, we have a $\mu$-boundary $(\Z^{\Z[1/2]/\Z}, \widetilde{\nu})$ where $\Z^{\Z[1/2]/\Z}$ is the space of functions from $\Z[1/2]/\Z$ to $\Z$. Pick any $y \in \Z[1/2]/\Z$ and a $k \in \Z$ so that the function defined by
	\(
	f(g) = \P_g\left[\cC_\infty(\mathbf{w})(y) = k \right]
	\) for each $g\in G$ satisfies $f(e_G) > 0$.
	
	Consider for every $n\ge 1$ the element $a_n\in T$ defined by
	\[
	a_n(x) = \begin{cases} y + 2^n(x - y) & \text{ if } y \leq x < y + 2^{-2n} \\
		y + 2^{-n} - 2^{-3n} + 2^{-n}(x-y) & \text{ if }y + 2^{-2n} < x < y + 2^{-2n} + 2^{-n}\\
		x & \text{ elsewhere,} \end{cases}
	\] see Figure \ref{fig:gn}.
	\begin{figure}[h!]
		\centering
		\begin{tikzpicture}[scale = 4, line width=0.6pt]
			\draw[-] (0,0) -- (0,1) -- (1,1) -- (1,0) -- (0,0);
			\node[below, scale=0.5] at (0,0) {\LARGE $0$};
			\node[below, scale=0.5] at (1,0) {\LARGE $1$};
			\node[left, scale=0.5] at (0,1) {\LARGE $1$};
			
			\draw[domain=0:1/2, smooth, variable=\x, red] plot ({\x}, {\x});
			\draw[domain=1/2:1/2 + 1/2^8, smooth, variable=\x, red] plot ({\x}, {1/2 + 16*(\x- 1/2)});
			\draw[domain=1/2 + 1/2^8:1/2 + 1/2^8 + 1/2^4, smooth, variable=\x, red] plot ({\x}, {1/2 + 1/2^4 - 1/2^9 + 1/2^4*(\x - 1/2)});
			\draw[domain=1/2 + 1/2^8 + 1/2^4:1, smooth, variable=\x, red] plot ({\x}, {\x});
			
			\draw[domain=0:1/2, smooth, variable=\x, green!60!black] plot ({\x}, {\x});
			\draw[domain=1/2:1/2 + 1/2^6, smooth, variable=\x, green!60!black] plot ({\x}, {1/2 + 8*(\x- 1/2)});
			\draw[domain=1/2 + 1/2^6:1/2 + 1/2^6 + 1/2^3, smooth, variable=\x, green!60!black] plot ({\x}, {1/2 + 1/2^3 - 1/2^9 + 1/2^3*(\x - 1/2)});
			\draw[domain=1/2 + 1/2^6 + 1/2^3:1, smooth, variable=\x, green!60!black] plot ({\x}, {\x});
			
			\draw[domain=0:1/2, smooth, variable=\x, blue] plot ({\x}, {\x});
			\draw[domain=1/2:1/2 + 1/16, smooth, variable=\x, blue] plot ({\x}, {1/2 + 4*(\x- 1/2)});
			\draw[domain=1/2 + 1/16:1/2 + 1/16 + 1/4, smooth, variable=\x, blue] plot ({\x}, {1/2 + 1/4 - 1/64 + 1/4*(\x - 1/2)});
			\draw[domain=1/2 + 1/16 + 1/4:1, smooth, variable=\x, blue] plot ({\x}, {\x});

		\end{tikzpicture}
		\caption{The maps $a_n$ for $y = 1/2$ and $n = 2$ (blue), $n = 3$ (green) and $n = 4$ (red).}
		\label{fig:gn}
	\end{figure}
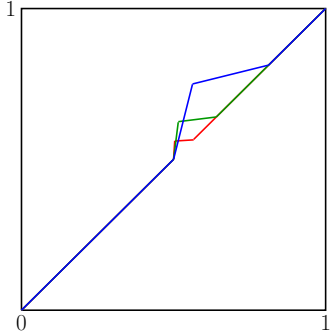
	
	Then, we have that
	\begin{itemize}
		\item $a_n (y) = y$,
		\item $\mathrm{supp}(a_n)$ is a dyadic interval containing $y$ and of length $2^{-n} + 2^{-2n}$, and
		\item the derivative jump of $a_n$ at $y$ is equal to $2^{n}$.
	\end{itemize}
	This implies that both $\mathrm{diam}(\mathrm{supp}(a_n))$ and $f(a_n)$ converge to $0$ as $n$ goes to infinity. From this point, one can continue just as in the proof of Theorem \ref{thm:lamps} to conclude that there is no $h \in L^\infty(S^1, \nu)$ such that
	\(
	f(g) = \int_{S^1}h(g(x)) \dd \nu(x)
	\) for all $g \in G$.
\end{rem}
\bibliographystyle{alpha}
\bibliography{biblio}
\end{document}